\documentclass[12pt]{amsart}
\usepackage{}

\usepackage{amsmath}
\usepackage{amsfonts}
\usepackage{amssymb}
\usepackage[all]{xy}           

\usepackage{bbding}
\usepackage{txfonts}
\usepackage{amscd}

\usepackage[shortlabels]{enumitem}
\usepackage{ifpdf}
\ifpdf
  \usepackage[colorlinks,final,backref=page,hyperindex]{hyperref}
\else
  \usepackage[colorlinks,final,backref=page,hyperindex,hypertex]{hyperref}
\fi
\usepackage{tikz}
\usepackage[active]{srcltx}
\usepackage{cancel}
\makeatletter

\newtheorem{thm}{Theorem}[section]
\newtheorem{lem}[thm]{Lemma}
\newtheorem{cor}[thm]{Corollary}
\newtheorem{pro}[thm]{Proposition}
\newtheorem{ex}[thm]{Example}
\newtheorem{rmk}[thm]{Remark}
\newtheorem{defi}[thm]{Definition}

\newcommand {\emptycomment}[1]{}

\newcommand{\lon }{\,\rightarrow\,}
\newcommand{\be }{\begin{equation}}
\newcommand{\ee }{\end{equation}}

\newcommand{\g}{\mathfrak g}
\newcommand{\h}{\mathfrak h}

\newcommand{\huaB}{\mathcal{B}}

\newcommand{\huaP}{\mathcal{P}}
\newcommand{\huaC}{{\mathcal{C}}}

\newcommand{\huaH}{\mathcal{H}}

\newcommand{\huaO}{{\mathcal{O}}}

\newcommand{\huaZ}{\mathcal{Z}}

\newcommand{\frkk}{\mathfrak k}

\newcommand{\frkT}{\mathfrak T}

\newcommand{\Courant}[1]{\left\llbracket  #1\right\rrbracket }

\newcommand{\Id}{{\rm{Id}}}

\newcommand{\br}[1]{   [ \cdot,    \cdot  ]   }

\newcommand{\Hom}{\mathrm{Hom}}

\newcommand{\gl}{\mathfrak {gl}}

\newcommand{\ad}{\mathrm{ad}}

\newcommand{\U}{\mathrm{U}}

\newcommand{\Ri}{\mathsf{R}}
\newcommand{\Li}{\mathsf{3Lie}}

\begin{document}

\title[$3$-post-Lie algebras and relative Rota-Baxter operators  on $3$-Lie algebras]{$3$-post-Lie algebras and relative Rota-Baxter operators of nonzero weight on $3$-Lie algebras}

\author{Shuai Hou}
\address{Department of Mathematics, Jilin University, Changchun 130012, Jilin, China}
\email{houshuai19@jlu.edu.cn}

\author{Yunhe Sheng}
\address{Department of Mathematics, Jilin University, Changchun 130012, Jilin, China}
\email{shengyh@jlu.edu.cn}

\author{Yanqiu Zhou}
\address{School of Science, Guangxi University of Science and Technology, Liuzhou 545006, China}
\email{zhouyanqiunihao@163.com}


\begin{abstract}
In this paper, first we introduce the notions of   relative Rota-Baxter operators of nonzero weight  on   $3$-Lie algebras  and  $3$-post-Lie algebras. A 3-post-Lie algebra consists of a 3-Lie algebra structure and a ternary operation such that some compatibility conditions are satisfied. We show that a relative Rota-Baxter operator of nonzero weight induces a $3$-post-Lie algebra naturally. Conversely, a $3$-post-Lie algebra  gives rise to a new 3-Lie algebra, which is called the subadjacent 3-Lie algebra, and an action on the original 3-Lie algebra. Then we construct an $L_\infty$-algebra  whose Maurer-Cartan elements are relative Rota-Baxter operators of  nonzero weight. Consequently, we obtain the twisted $L_\infty$-algebra that controls deformations of a given relative Rota-Baxter operator of  nonzero weight on 3-Lie algebras. Finally, we introduce  a cohomology theory for a relative Rota-Baxter operator of nonzero weight on $3$-Lie algebras and use the second cohomology group to classify infinitesimal deformations.
\end{abstract}


\maketitle

\tableofcontents

\allowdisplaybreaks


\section{Introduction}
Rota-Baxter associative algebras originated from the probability
study of G. Baxter~\cite{Ba}. In the Lie algebra context,   Kupershmidt introduced the notion of an $\huaO$-operator (also called
a relative Rota-Baxter operator)
in \cite{Ku} to better understand the classical Yang-Baxter equation.
(Relative) Rota-Baxter operators on Lie algebras and associative algebras have important applications in various fields, such as the classical Yang-Baxter equation and integrable systems \cite{Bai,Ku,STS}, splitting of operads \cite{BBGN,PBG}, double Lie algebras \cite{GK},  Connes-Kreimer's algebraic approach to renormalization
of quantum  field theory \cite{CK}, etc.
See the book \cite{Gub} for more details. Recently, the deformation and cohomology theories of relative Rota-Baxter operators on both Lie and associative algebras were studied in \cite{Das,TBGS,WaiZhou}.
  Post-Lie algebras, as natural generalizations of pre-Lie algebras \cite{Burde-1},  were introduced by  Vallette in \cite{Val}, and  have important  applications in geometric numerical integration and mathematical physics \cite{BaiGuoNi,Burde-Moens-2,Fard1,Fard2,Ebrahimi,Kaas}.   In particular, a relative Rota-Baxter operator of nonzero weight on Lie algebras  induces a post-Lie algebra.


The notion of  $3$-Lie algebras and more generally, $n$-Lie algebras was introduced by Filippov in \cite{Filippov}, which appeared naturally in various areas of theoretical and mathematical physics \cite{Bagger-J,dMFM,FO}. See the review article \cite{review,Makhlouf} for more details.
Indeed, the $n$-Lie algebra is the algebraic structure corresponding to the Nambu mechanics  \cite{Nambu}. In \cite{BaiRGuo}, R. Bai et al. introduced the notion of   Rota-Baxter operators of weight $\lambda$ on   $3$-Lie algebras and give various constructions from Rota-Baxter operators  on  Lie algebras  and pre-Lie algebras. Then the notion of an $\huaO$-operator on a $3$-Lie algebra with respect to a representation was introduced in \cite{BGS-3-Bialgebras}   to study   solutions of the $3$-Lie classical Yang-Baxter equation. In particular, when the representation is the adjoint representation, an $\huaO$-operator is exactly a Rota-Baxter operator  of weight $0.$ $\huaO$-operators on 3-Lie algebras were also used to study matched pairs of 3-Lie algebras and 3-Lie bialgebras in \cite{HST}.
Cohomologies and deformations of $\huaO$-operators on 3-Lie algebras   were studied in \cite{THS}, where the terminology of relative Rota-Baxter operators was used instead of $\huaO$-operators.

The first purpose of this paper is to study relative Rota-Baxter operators of nonzero weight on 3-Lie algebras and associated structures. For this purpose,  first we study actions of a 3-Lie algebra $\g$ on a  3-Lie algebra $\h$, which is totally different from the case of Lie algebra actions. We  introduce the notion of a relative Rota-Baxter operator of weight $\lambda$ from a $3$-Lie algebra $\h$ to a $3$-Lie algebra $\g$ with respect to an action $\rho$, and characterize it using graphs of the semidirect product 3-Lie algebra. We further establish the deformation and cohomology theories for  relative Rota-Baxter operators of weight $\lambda$ on $3$-Lie algebras. Note that relative Rota-Baxter operators and Rota-Baxter operators of nonzero weight introduced in \cite{BaiRGuo} are not consistent, see Remark \ref{rmk:difference} for details of the explanation. The  cohomology theory for  Rota-Baxter operators of weight $\lambda$ on $3$-Lie algebras is considered in \cite{GQWZ} separately.

The second purpose is to investigate the $3$-ary generalizations of post-Lie algebras and analyze the relation with the aforementioned relative Rota-Baxter operators of nonzero weight on 3-Lie algebras. We introduce a new algebraic structure, which is called a $3$-post-Lie algebra. A $3$-post-Lie algebra naturally gives rise to a new $3$-Lie algebra and an action on the original 3-Lie algebra such that the identity map is a relative Rota-Baxter operator of weight 1. We show that a relative Rota-Baxter operator of weight $\lambda$  induces a $3$-post-Lie algebra structure naturally. We will explore further applications of 3-post-Lie algebras along the line of applications of post-Lie algebras in future works. 

The paper is organized as follows. In Section \ref{sec:two}, we introduce the notion of a relative Rota-Baxter operator of weight $\lambda$ from a $3$-Lie algebra $\h$ to a $3$-Lie algebra $\g$ with respect to an action $\rho$. In Section \ref{sec:three}, we introduce the notion of
a $3$-post-Lie algebra and show that a relative Rota-Baxter operator  of weight $\lambda$ on  $3$-Lie algebras naturally
induces a $3$-post-Lie algebra. Moreover, a $3$-post-Lie algebra also gives rise to a new 3-Lie algebra together with an action.
In Section \ref{sec:four}, we construct an $L_{\infty}$-algebra whose  Maurer-Cartan elements  are precisely   relative Rota-Baxter operators of weight $\lambda$ on   $3$-Lie algebras.
In Section \ref{sec:five}, we establish a cohomology theory for a relative Rota-Baxter operator of weight $\lambda$ on   $3$-Lie algebras, and classify infinitesimal deformations using the second cohomology group.

\vspace{2mm}

In this paper, we work over an algebraically closed filed $\mathbb K$ of characteristic $0$.


\vspace{2mm}
\noindent
{\bf Acknowledgements.} This research is  supported by NSFC
(11922110). We give warmest thanks to the referee for helpful suggestions.

\section{Relative Rota-Baxter operators of weight $\lambda$ on $3$-Lie algebras}\label{sec:two}

In this section, we first introduce the notion of action of 3-Lie algebras, which give rise to the semidirect product 3-Lie algebras. Then we introduce  the notion of relative Rota-Baxter operators of weight $\lambda$ on $3$-Lie algebras, which can be characterized by the graphs of the  semidirect product 3-Lie algebras. Finally we establish the relation between relative Rota-Baxter operators of weight $\lambda$ on $3$-Lie algebras and Nijenhuis operators on 3-Lie algebras. A class of examples are given via certain projections.
\begin{defi}{\rm (\cite{Filippov})}\label{defi:3Lie}
A {\bf 3-Lie algebra}
is a vector space $\g$ together with a skew-symmetric linear map $[\cdot,\cdot,\cdot]_{\g}:\wedge^{3}\g\rightarrow \g$, such that for $ x_{i}\in \g, 1\leq i\leq 5$, the following {\bf Fundamental Identity} holds:
\begin{eqnarray}
\nonumber\qquad &&[x_1,x_2,[x_3,x_4, x_5]_{\g}]_{\g}\\
&=&[[x_1,x_2, x_3]_{\g},x_4,x_5]_{\g}+[x_3,[x_1,x_2, x_4]_{\g},x_5]_{\g}+[x_3,x_4,[x_1,x_2, x_5]_{\g}]_{\g}.
 \label{eq:jacobi1}
\end{eqnarray}
\end{defi}
For $x_{1},x_{2}\in \g$, define $\ad_{x_1,x_2}\in \gl(\g)$ by
\begin{eqnarray*}\label{eq2}
\ad_{x_{1},x_{2}}x:=[x_{1},x_{2},x]_{\g},\quad \forall x\in \g.
\end{eqnarray*}
Then $\ad_{x_{1},x_{2}}$ is a derivation, i.e.
$$\ad_{x_{1},x_{2}}[x_{3},x_{4},x_{5}]_{\g}=[\ad_{x_{1},x_{2}}x_{3},x_{4},x_{5}]_{\g}+
[x_{3},\ad_{x_{1},x_{2}}x_{4},x_{5}]_{\g}+[x_{3},x_{4},\ad_{x_{1},x_{2}}x_{5}]_{\g}.$$
\begin{defi}{\rm (\cite{KA})}
A {\bf representation} of a $3$-Lie algebra $(\g,[\cdot,\cdot,\cdot]_{\g})$ on a vector space $V$ is a linear
map: $\rho:\wedge^{2}\g\rightarrow \gl(V)$, such that for all $x_{1}, x_{2}, x_{3}, x_{4}\in \g,$ the following equalities hold:
\begin{eqnarray}
~\label{representation-1}\rho(x_{1},x_{2})\rho(x_{3},x_{4})&=&\rho([x_{1},x_{2},x_{3}]_{\g},x_{4})+
\rho(x_{3},[x_{1},x_{2},x_{4}]_{\g})+\rho(x_{3},x_{4})\rho(x_{1},x_{2});\\
~\label{representation-2}\rho(x_{1},[x_{2},x_{3},x_{4}]_{\g})&=&\rho(x_{3},x_{4})\rho(x_{1},x_{2})-\rho(x_{2},x_{4})\rho(x_{1},x_{3})
+\rho(x_{2},x_{3})\rho(x_{1},x_{4}).
\end{eqnarray}
\end{defi}

Let $(\g,[\cdot,\cdot,\cdot]_{\g})$ be a $3$-Lie algebra.  The linear map $\ad:\wedge^2\g\rightarrow\gl(\g)$ defines a representation
of the $3$-Lie algebra $\g$ on itself, which is called the {\bf adjoint representation} of $\g.$

\begin{defi}{\rm (\cite{Filippov})}
Let $(\g,[\cdot,\cdot,\cdot]_{\g})$ be a $3$-Lie algebra. Then the subalgebra $[\g,\g,\g]_{\g}$ is called the {\bf derived algebra} of $\g$, and denoted by ${\g}^1$. 
The subspace $$\huaC(\g)=\{x\in \g~|~[x,y,z]_{\g}=0,~ \forall y,z\in\g\}$$   is called the {\bf center} of $\g$.
\end{defi}

\begin{defi}
Let $(\g,[\cdot,\cdot,\cdot]_{\g})$ and $(\h,[\cdot,\cdot,\cdot]_{\h})$  be two $3$-Lie algebras.
Let $\rho: \wedge^2\g\rightarrow \gl(\h)$ be a representation of the $3$-Lie algebra $\g$ on the vector space $\h$.
If
for all $x,y\in \g, u,v,w\in \h,$
\begin{eqnarray}
\label{eq:action-1}{}\rho(x,y)u\in \huaC(\h),\\
\label{eq:action-2}{}\rho(x,y)[u,v,w]_{\h}=0,
\end{eqnarray}
then $\rho$ is called  an {\bf action} of  $\g$ on  $\h.$
\end{defi}

We denote an action by $(\h;\rho).$ Note that \eqref{eq:action-1} and \eqref{eq:action-2} imply that $\rho(x,y)$ is a derivation.

\begin{ex}
Let $(\g,[\cdot,\cdot,\cdot]_{\g})$ be a $3$-Lie algebra. If $\g$ satisfies ${\g}^{1}\subset \huaC(\g),$ then the adjoint representation $\ad:\wedge^2\g\rightarrow\gl(\g)$ is an action of $\g$ on itself.
\end{ex}

\begin{defi}\label{defi:rb-3-lie-algebra}
Let $\rho: \wedge^2\g\rightarrow \gl(\h)$ be an action of a $3$-Lie algebra $(\g,[\cdot,\cdot,\cdot]_{\g})$ on a $3$-Lie algebra $(\h,[\cdot,\cdot,\cdot]_{\h}).$
A linear map $T: \h\rightarrow\g$ is called a {\bf
  relative Rota-Baxter operator}  of weight $\lambda\in \mathbb K$ from a $3$-Lie algebra $\h$ to a $3$-Lie algebra $\g$ with respect to an action $\rho$  if
\begin{eqnarray} \label{eq:rRB}
\qquad [Tu,Tv,Tw]_\g
=T\Big(\rho(Tu,Tv)w+\rho(Tv,Tw)u+\rho(Tw,Tu)v+\lambda[u,v,w]_\h\Big),\quad \forall u, v, w\in\h.
\end{eqnarray}
\end{defi}

\begin{rmk}\label{rmk:difference}

 The  notion of   a relative Rota-Baxter operator  of weight $\lambda$ on a $3$-Lie algebra given above is a natural generalization of the $\huaO$-operator introduced in {\rm \cite{BGS-3-Bialgebras}}. While it is not consistent with the Rota-Baxter operator of weight $\lambda$ introduced in \cite{BaiRGuo}.
Therefore, there are two different theories for   relative Rota-Baxter operators of weight $\lambda$ introduced above  and Rota-Baxter operators of weight $\lambda$ on $3$-Lie algebras introduced in \cite{BaiRGuo}. We will see that relative Rota-Baxter operators of weight $\lambda$ on $3$-Lie algebras can be characterized by Maurer-Cartan elements of the controlling $L_\infty$-algebras, and naturally related to $3$-post-Lie algebras and Nijenhuis structures. These facts can be viewed as justifications of relative Rota-Baxter operators of weight $\lambda$ on $3$-Lie algebras being interesting structures.
\end{rmk}

\begin{defi}
Let $T$ and $T'$ be two relative Rota-Baxter operators of weight $\lambda$ from a $3$-Lie algebra $(\h,[\cdot,\cdot,\cdot]_\h)$ to a $3$-Lie algebra $(\g,[\cdot,\cdot,\cdot]_\g)$ with respect to an action $\rho$. A {\bf homomorphism} from $T$ to $T'$ consists of $3$-Lie algebra homomorphisms $\psi_\g: \g\lon\g$ and   $\psi_\h: \h\lon\h$ such that
\begin{eqnarray}
 \label{condition-1}\psi_\g\circ T&=&T'\circ\psi_\h,\\
  \label{condition-2}\psi_\h(\rho(x,y)u)&=&\rho(\psi_\g(x),\psi_\g(y))(\psi_\h(u)),\quad \forall x,y\in \g, u\in \h.
\end{eqnarray}
In particular, if both $\psi_\g$ and $\psi_\h$ are invertible, $(\psi_\g, \psi_\h)$ is called an {\bf isomorphism} from $T$ to $T'$.
\end{defi}

Let $\rho: \wedge^2\g\rightarrow \gl(\h)$ be an action of a $3$-Lie algebra $(\g,[\cdot,\cdot,\cdot]_{\g})$ on a $3$-Lie algebra $(\h,[\cdot,\cdot,\cdot]_{\h}).$ Define a skew-symmetric bracket operation $[\cdot,\cdot,\cdot]_{\rho}$ on $\g\oplus\h$ by
\begin{eqnarray}
{}[x+u,y+v,z+w]_{\rho}=[x,y,z]_{\g}+\rho(x,y)w+\rho(y,z)u+\rho(z,x)v+\lambda[u,v,w]_\h,
\end{eqnarray}
 for all $x,y,z\in\g,~u,v,w\in\h.$

\begin{pro}\label{pro:semi}
 With above notations,  $(\g\oplus \h,[\cdot,\cdot,\cdot]_{\rho})$ is a $3$-Lie algebra, which is called the semidirect product of the $3$-Lie algebra $\g$ and the $3$-Lie algebra $\h$ with respect to the action $\rho$, and denoted by $\g\ltimes _\rho\h.$
\end{pro}
\begin{proof}
  It follows from straightforward computations, and we omit details.
\end{proof}

\begin{thm}\label{graphro}
Let $\rho: \wedge^2\g\rightarrow \gl(\h)$ be an action of a $3$-Lie algebra $(\g,[\cdot,\cdot,\cdot]_{\g})$ on a $3$-Lie algebra $(\h,[\cdot,\cdot,\cdot]_{\h}).$  Then a linear map $T: \h\rightarrow\g$ is a relative Rota-Baxter operator of weight $\lambda$ if and only if the graph   $$Gr(T)=\{Tu+u|u\in \h\}$$ is a subalgebra of the $3$-Lie algebra $\g\ltimes _\rho\h$.
\end{thm}
\begin{proof}
Let $T:\h\rightarrow\g$ be a linear map.
For all $u,v,w\in \h,$  we have
\begin{eqnarray*}
&&[Tu+u,Tv+v,Tw+w]_{\rho}\\
&=&[Tu,Tv,Tw]_{\g}+\rho(Tu,Tv)w+\rho(Tv,Tw)u+\rho(Tw,Tu)v+\lambda[u,v,w]_{\h},
\end{eqnarray*}
which implies that the graph $Gr(T)=\{Tu+u|u\in \h\}$ is a subalgebra of the $3$-Lie algebra $\g\ltimes _\rho\h$ if and only if $T$ satisfies
\begin{eqnarray*}
[Tu,Tv,Tw]_{\g}=T\Big(\rho(Tu,Tv)w+\rho(Tv,Tw)u+\rho(Tw,Tu)v+\lambda[u,v,w]_{\h}\Big),
\end{eqnarray*}
which means that $T$ is a relative Rota-Baxter operator of weight $\lambda$.
\end{proof}
Since the graph $Gr(T)$ is isomorphic to $\h$ as a vector space, there is an induced $3$-Lie algebra structure on $\h.$
\begin{cor}\label{new3-liealg}
Let $T: \h\rightarrow\g$ be a relative Rota-Baxter operator of weight $\lambda$ from a $3$-Lie algebra $(\h,[\cdot,\cdot,\cdot]_\h)$ to a $3$-Lie algebra $(\g,[\cdot,\cdot,\cdot]_\g)$ with respect to an action  $\rho$. Then $(\h, [\cdot,\cdot,\cdot]_T)$ is a $3$-Lie algebra, called {\bf the descendent $3$-Lie algebra} of $T$, where
\begin{equation}\label{eq:des3-Lieb}
[u,v,w]_T=\rho(Tu,Tv)w+\rho(Tv,Tw)u+\rho(Tw,Tu)v+\lambda[u,v,w]_{\h}, \quad \forall u,v,w\in\h.
\end{equation}
Moreover, $T$ is a $3$-Lie algebra homomorphism from $(\h, [\cdot,\cdot,\cdot]_T)$ to $(\g, [\cdot,\cdot,\cdot]_\g)$.
\end{cor}

In the sequel, we give the relationship between relative Rota-Baxter operators of weight $\lambda$ and Nijenhuis operators.
Recall from \cite{Liu-Jie-Feng} that a Nijenhuis operator on a $3$-Lie algebra $(\g,[\cdot,\cdot,\cdot]_{\g})$ is a linear map $N:\g\rightarrow\g$ satisfying
\begin{eqnarray}\label{Nijenhuis-1}
[Nx,Ny,Nz]_{\g}&=&N\Big([Nx,Ny,z]_{\g}+[x,Ny,Nz]_{\g}+[Nx,y,Nz]_{\g}\\
\nonumber&&-N[Nx,y,z]_{\g}-N[x,Ny,z]_{\g}-N[x,y,Nz]_{\g}\\
\nonumber&&+N^2[x,y,z]_{\g}\Big), \quad \forall  x,y,z\in \g.
\end{eqnarray}

\begin{pro}
Let $\rho: \wedge^2\g\rightarrow \gl(\h)$ be an action of a $3$-Lie algebra $(\g,[\cdot,\cdot,\cdot]_{\g})$ on a $3$-Lie algebra $(\h,[\cdot,\cdot,\cdot]_{\h}).$  Then a linear map $T: \h\rightarrow\g$ is a relative Rota-Baxter operator of weight $\lambda$ if and only if
\begin{equation*}
\overline{T}=
\left(\begin{array}{cc}
{\Id}&T\\
0&0\\
\end{array}\right):\g\oplus\h\rightarrow\g\oplus\h,
\end{equation*}
is a Nijenhuis operator acting on the semidirect product $3$-Lie algebra $\g\ltimes _\rho\h.$
\end{pro}
\begin{proof}
For all $x,y,z\in\g,u,v,w\in \h,$ on the one hand, we have
\begin{eqnarray*}
[\overline{T}(x+u),\overline{T}(y+v),\overline{T}(z+w)]_{\rho}&=&[x,y,z]_{\g}+[Tu,y,z]_{\g}+[x,Tv,z]_{\g}+[x,y,Tw]_{\g}\\
&&+[x,Tv,Tw]_{\g}+[Tu,y,Tw]_{\g}+[Tu,Tv,z]_{\g}+[Tu,Tv,Tw]_{\g}.
\end{eqnarray*}
On the other hand, since $\overline{T}^2=\overline{T},$ we have
\begin{eqnarray*}
&&\overline{T}\Big([\overline{T}(x+u),\overline{T}(y+v),z+w]_{\rho}+[\overline{T}(x+u),y+v,\overline{T}(z+w)]_{\rho}+[x+u,\overline{T}(y+v),\overline{T}(z+w)]_{\rho}\Big)\\
&&-\overline{T}^2\Big([\overline{T}(x+u),y+v,z+w]_{\rho}+[x+u,\overline{T}(y+v),z+w]_{\rho}+[x+u,y+v,\overline{T}(z+w)]_{\rho}\Big)\\
&&+\overline{T}^3\Big([x+u,y+v,z+w]_{\rho}\Big)\\
&=&[x,y,z]_{\g}+[Tu,y,z]_{\g}+[x,Tv,z]_{\g}+[x,y,Tw]_{\g}+[x,Tv,Tw]_{\g}+[Tu,y,Tw]_{\g}+[Tu,Tv,z]_{\g}\\
&&+T\Big(\rho(Tu,Tv)w+\rho(Tv,Tw)u+\rho(Tw,Tu)v+\lambda[u,v,w]_{\h}\Big),
\end{eqnarray*}
which implies that $\overline{T}$ is a Nijenhuis operator on the semidirect product $3$-Lie algebra $\g\ltimes _\rho\h$ if and only if \eqref{eq:rRB} is satisfied.
\end{proof}

At the end of this section, we show that certain projections provide a class of examples of relative Rota-Baxter operators on 3-Lie algebras.

\begin{pro}\label{pro:pro}
Let $(\g,[\cdot,\cdot,\cdot]_\g)$ be a $3$-Lie algebra such that the adjoint representation $\ad:\wedge^2\g\rightarrow\gl(\g)$
is an action of the $3$-Lie algebra $\g$ on itself. Let $\h$ be an abelian $3$-Lie subalgebra of $\g$ satisfying $\g^1\cap\h=0.$
Let $\frkk$ be a complement of $\h$ such that $\g=\frkk\oplus \h$ as a vector space. Then the projection $P:\g\rightarrow\g$  onto the subspace $\h$, i.e. $P(k,u)=u$ for all $k\in\frkk,~u\in\h$,
is a relative Rota-Baxter operator of weight $\lambda$ from $\g$ to $\g$ with respect to the adjoint action $\ad$.
\end{pro}
\begin{proof}
For all $x,y,z\in \g,$ denote by $\overline{x},\overline{y},\overline{z}$ their images under the projection $P$. Since $\h$ is abelian and $\g^1\cap\h=0,$ we have
\begin{eqnarray*}
&&[P(x),P(y),P(z)]_\g-P\Big([P(x),P(y),z]_\g+[x,P(y),P(z)]_\g+[P(x),y,P(z)]_\g+\lambda[x,y,z]_\g\Big)\\
&=&[\overline{x},\overline{y},\overline{z}]_\g-P\Big([\overline{x},\overline{y},z]_\g+[\overline{x},y,\overline{z}]_\g+[x,\overline{y},\overline{z}]_\g+\lambda[x,y,z]_\g\Big)\\
&=&0,
\end{eqnarray*}
which implies that the projection $P$ is a relative Rota-Baxter operator of weight $\lambda$ from $\g$ to $\g$ with respect to the adjoint action $\ad$.
\end{proof}

\begin{ex}
Let $(\g,[\cdot,\cdot,\cdot]_{\g})$ be a $4$-dimensional $3$-Lie algebra with a basis $\{e_1,e_2,e_3,e_4\}$ and the nonzero multiplication is given by $$[e_2,e_3,e_4]=e_1.$$ The center of $\g$ is the subspace generated by $\{e_1\}.$ It is obvious that
  the adjoint representation $\ad:\wedge^2\g\rightarrow\gl(\g)$
is an action of  $\g$ on itself. Let $\h$ be a subalgebra of $\g$  generated by $\{e_3,e_4\}$. By Proposition \ref{pro:pro},   the projection $P:\g\rightarrow\g$ given by
$\begin{cases}
P(e_1)=0,\\
P(e_2)=0,\\
P(e_3)=e_3,\\
P(e_4)=e_4,
\end{cases}$
is a relative Rota-Baxter operator of weight $\lambda$ from $\g$ to $\g$ with respect to the adjoint action $\ad$.
\end{ex}

\emptycomment{
\begin{ex}
Let $(\g,[\cdot,\cdot,\cdot]_{\g})$ be a $4$-dimensional $3$-Lie algebra with a basis $\{e_1,e_2,e_3,e_4\}$ and the nonzero multiplication is given by $$[e_2,e_3,e_4]=e_1.$$
The center of $\g$ is the subspace generated by $\{e_1\}.$ It obviously  that
  the adjoint representation $\ad:\wedge^2\g\rightarrow\gl(\g)$
is an action of a $3$-Lie algebra $\g$ on itself.
Then   $T=\left(\begin{array}{cccc}
a_{11}&a_{12}&a_{13}&a_{14}\\
a_{21}&a_{22}&a_{23}&a_{23}\\
a_{31}&a_{32}&a_{33}&a_{34}\\
a_{41}&a_{42}&a_{43}&a_{44}
\end{array}\right)$ is a relative Rota-Baxter operator of weight $1$ from $\g$ to $\g$ with respect to the action $\ad$ if and only if
\begin{eqnarray*}
\label{rota-baxter-3-lie-ex}[Te_i,Te_j,Te_k]=T\Big([Te_i,Te_j,e_k]+[Te_i,e_j,Te_k]+[e_i,Te_j,Te_k]+[e_i,e_j,e_k]\Big),\quad i,j,k=1,2,3,4.
\end{eqnarray*}

In particular,
$T=\left(\begin{array}{cccc}
-1&0&0&0\\
0&0&0&0\\
0&0&0&-1\\
0&0&-1&0
\end{array}\right)$ and
$T=\left(\begin{array}{cccc}
1&0&0&0\\
0&0&1&0\\
0&1&0&0\\
0&0&0&0
\end{array}\right)$ are relative Rota-Baxter operators of weight $1$ from $\g$ to $\g$ with respect to the action $\ad$.
\end{ex}
}

\section{$3$-post-Lie algebras}\label{sec:three}
In this section, we introduce the notion of $3$-post-Lie algebras. A relative Rota-Baxter operator of nonzero weight induces a $3$-post-Lie algebra naturally. Therefore, $3$-post-Lie algebras can be viewed as the underlying
algebraic structures of relative Rota-Baxter operators of weight $\lambda$ on $3$-Lie algebras. A $3$-post-Lie algebra also gives rise to a new 3-Lie algebra and an action on the original 3-Lie algebra, which leads to the fact that the identity map is a relative Rota-Baxter operator of weight 1.
\begin{defi}\label{defi-3-post-Lie-algebra}
A {\bf $3$-post-Lie algebra} $(A,\{\cdot,\cdot,\cdot\},[\cdot,\cdot,\cdot])$ consists of a $3$-Lie algebra $(A,[\cdot,\cdot,\cdot])$ and a ternary product $\{\cdot,\cdot,\cdot\}:A\otimes A\otimes A\rightarrow A$ such that
\begin{eqnarray}
\label{3-Post-Lie-1} \{x_1,x_2,x_3\}&=&-\{x_2,x_1,x_3\},\\
\label{3-Post-Lie-2}\{x_1,x_2,\{x_3,x_4,x_5\}\}&=&\{x_3,x_4,\{x_1,x_2,x_5\}\}+\{\Courant{x_1,x_2,x_3},x_4,x_5\}+\{x_3,\Courant{x_1,x_2,x_4},x_5\},\\
\label{3-Post-Lie-3}\qquad\{\Courant{x_1,x_2,x_3},x_4,x_5\}&=& \{x_1,x_2,\{x_3,x_4,x_5\}\}+\{x_2,x_3,\{x_1,x_4,x_5\}\}+\{x_3,x_1,\{x_2,x_4,x_5\}\},\\
\label{3-Post-Lie-4}\{x_1,x_2,[x_3,x_4,x_5]\}&=&0,\\
\label{3-Post-Lie-5}[\{x_1,x_2,x_3\},x_4,x_5]&=&0,
\end{eqnarray}
for all $x_i\in A, 1\leq i\leq5$ and $\Courant{\cdot,\cdot,\cdot}$ is defined by
\begin{eqnarray}
\label{3-post-Lie} \Courant{x,y,z}=\{x,y,z\}+\{y,z,x\}+\{z,x,y\}+[x,y,z],\quad \forall x,y,z\in A.
\end{eqnarray}
\end{defi}

\begin{rmk}
Let $(A,\{\cdot,\cdot,\cdot\},[\cdot,\cdot,\cdot])$ be a $3$-post-Lie algebra. If the $3$-Lie bracket $[\cdot,\cdot,\cdot]=0$, then $(A,\{\cdot,\cdot,\cdot\})$  becomes a $3$-pre-Lie algebra which was introduced in \cite{BGS-3-Bialgebras} in the study of the $3$-Lie Yang-Baxter equation.
\end{rmk}

\begin{defi}
A  homomorphism from a $3$-post-Lie algebra $(A,\{\cdot,\cdot,\cdot\}, [\cdot,\cdot,\cdot])$ to $(A',\{\cdot,\cdot,\cdot\}', [\cdot,\cdot,\cdot]')$ is a linear map $\psi:A\lon A'$ satisfying
\begin{eqnarray}
\psi(\{x,y,z\})&=&\{\psi(x),\psi(y),\psi(z)\}',\\
\psi([x,y,z])&=&[\psi(x),\psi(y),\psi(z)]',
\end{eqnarray}
for all $x,y,z\in A.$
\end{defi}

\begin{thm}
Let $(A,\{\cdot,\cdot,\cdot\},[\cdot,\cdot,\cdot])$ be a $3$-post-Lie algebra. Then
\begin{itemize}
  \item[{\rm (i)}]  $(A, \Courant{\cdot,\cdot,\cdot})$
is a $3$-Lie algebra, which is called the {\bf sub-adjacent $3$-Lie algebra} of the $3$-post Lie algebra $(A,\{\cdot,\cdot,\cdot\},[\cdot,\cdot,\cdot])$, where $\Courant{\cdot,\cdot,\cdot}$ is defined by
\eqref{3-post-Lie}.
 \item[{\rm (ii)}]  $(A;L)$ is an action of the $3$-Lie algebra $(A,\Courant{\cdot,\cdot,\cdot})$ on the $3$-Lie algebra $(A,[\cdot,\cdot,\cdot])$, where the action $L:\otimes^2 A\rightarrow \gl(A)$ is defined by
$$ L(x,y)z=\{x,y,z\},\quad \forall~x,y,z\in A.$$

 \item[{\rm (iii)}] The identity map $\Id:A\to A$ is a relative Rota-Baxter operator of weight $1$ from the $3$-Lie algebra $(A,[\cdot,\cdot,\cdot])$ to the $3$-Lie algebra $(A,\Courant{\cdot,\cdot,\cdot})$ with respect to the action $L.$
\end{itemize}
\end{thm}

\begin{proof}

(i) By \eqref{3-Post-Lie-1}, it is obvious that the operation $\Courant{\cdot,\cdot,\cdot}$ given by \eqref{3-post-Lie} is skew-symmetric.
For all $x_1,x_2,x_3,x_4,x_5\in A,$ by \eqref{3-Post-Lie-2}-\eqref{3-Post-Lie-5}, we have
\begin{eqnarray*}
&&\Courant{x_1,x_2,\Courant{x_3,x_4,x_5}}-\Courant{\Courant{x_1,x_2,x_3},x_4,x_5}\\
&&-\Courant{x_3,\Courant{x_1,x_2,x_4},x_5}-\Courant{x_3,x_4,\Courant{x_1,x_2,x_5}}\\
&=&\{x_1,x_2,\{x_3,x_4,x_5\}\}+\{x_1,x_2,\{x_4,x_5,x_3\}\}+\{x_1,x_2,\{x_5,x_3,x_4\}\}\\
&&+\{x_1,x_2,[x_3,x_4,x_5]\}+\{\Courant{x_3,x_4,x_5},x_1,x_2\}+\{x_2,\Courant{x_3,x_4,x_5},x_1\}\\
&&+[x_1,x_2,\Courant{x_3,x_4,x_5}]-\{x_4,x_5,\{x_1,x_2,x_3\}\}-\{x_4,x_5,\{x_2,x_3,x_1\}\}\\
&&-\{x_4,x_5,\{x_3,x_1,x_2\}\}-\{x_4,x_5,[x_1,x_2,x_3]\}-\{\Courant{x_1,x_2,x_3},x_4,x_5\}\\
&&-\{x_5,\Courant{x_1,x_2,x_3},x_4\}-[\Courant{x_1,x_2,x_3},x_4,x_5]-\{x_5,x_3,\{x_1,x_2,x_4\}\}\\
&&-\{x_5,x_3,\{x_2,x_4,x_1\}\}-\{x_5,x_3,\{x_4,x_1,x_2\}\}-\{x_5,x_3,[x_1,x_2,x_4]\}\\
&&-\{x_3,\Courant{x_1,x_2,x_4},x_5\}-\{\Courant{x_1,x_2,x_4},x_5,x_3\}-[x_3,\Courant{x_1,x_2,x_4},x_5]\\
&&-\{x_3,x_4,\{x_1,x_2,x_5\}\}-\{x_3,x_4,\{x_5,x_1,x_2\}\}-\{x_3,x_4,\{x_2,x_5,x_1\}\}\\
&&-\{x_3,x_4,[x_1,x_2,x_5]\}-\{\Courant{x_1,x_2,x_5},x_3,x_4\}-\{x_4,\Courant{x_1,x_2,x_5},x_3\}\\
&&-[x_3,x_4,\Courant{x_1,x_2,x_5}]\\
&=&0,
\end{eqnarray*}
which implies that $(A,\Courant{\cdot,\cdot,\cdot})$ is a $3$-Lie algebra.

(ii) By \eqref{3-Post-Lie-2} and \eqref{3-Post-Lie-3}, we have
 \begin{eqnarray*}
[L(x_1,x_2),L(x_3,x_4)](x_5)&=&L(\Courant{x_1,x_2,x_3},x_4)(x_5)+L(x_3,\Courant{x_1,x_2,x_4})(x_5),\\
 L(\Courant{x_1,x_2,x_3},x_4)(x_5)&=&\Big(L(x_1,x_2)L(x_3,x_4)+L(x_2,x_3)L(x_1,x_4)\\&&+L(x_3,x_1)L(x_2,x_4)\Big)(x_5),
 \end{eqnarray*}
which implies that $L$ is a representation of the 3-Lie algebra $(A,\Courant{\cdot,\cdot,\cdot})$ on the vector space $A$. By \eqref{3-Post-Lie-5}, $L(x_1,x_2)x_3$ is in the center of the $3$-Lie algebra $(A,[\cdot,\cdot,\cdot])$. Then by \eqref{3-Post-Lie-4}, we have $L(x_1,x_2)[x_3,x_4,x_5]=0.$
 Therefore, $(A;L)$ is an action of the $3$-Lie algebra $(A,\Courant{\cdot,\cdot,\cdot})$ on the $3$-Lie algebra $(A,[\cdot,\cdot,\cdot])$.

 (iii) It follows from the definition of the subadjacent 3-Lie bracket $\Courant{\cdot,\cdot,\cdot}$ directly.
\end{proof}

\begin{cor}\label{cor:mor-post-3-Lie}
 Let $\psi:A\lon A'$ be  a  homomorphism from a $3$-post-Lie algebra $(A,\{\cdot,\cdot,\cdot\}, [\cdot,\cdot,\cdot])$ to $(A',\{\cdot,\cdot,\cdot\}', [\cdot,\cdot,\cdot]')$. Then $\psi$ is also a  homomorphism from the sub-adjacent $3$-Lie algebra $(A, \Courant{\cdot,\cdot,\cdot})$ to $(A', \Courant{\cdot,\cdot,\cdot}')$.
\end{cor}

The following results illustrate that $3$-post-Lie algebras can be viewed as the underlying algebraic structures of relative Rota-Baxter operators of weight $\lambda$ on $3$-Lie algebras.

\begin{thm}\label{construct-3-Post-Lie algebra}
Let $T:\h\rightarrow\g$ be a relative Rota-Baxter operator of weight $\lambda$ from a $3$-Lie algebra $(\h,[\cdot,\cdot,\cdot]_\h)$ to a $3$-Lie algebra $(\g,[\cdot,\cdot,\cdot]_\g)$ with respect to an action $\rho$.
\begin{itemize}
\item[{\rm (i)}] $(\h,\{\cdot,\cdot,\cdot\},[\cdot,\cdot,\cdot])$ is a $3$-post-Lie algebra, where $\{\cdot,\cdot,\cdot\}$ and $[\cdot,\cdot,\cdot]$ are given by
\begin{eqnarray}\label{twisted-post}
\qquad\{u_1,u_2,u_3\}=\rho(Tu_1,Tu_2)u_3,\quad [u_1,u_2,u_3]=\lambda[u_1,u_2,u_3]_\h, \quad \forall u_1,u_2,u_3 \in \h.
\end{eqnarray}
\item[{\rm (ii)}] $T$ is a $3$-Lie algebra homomorphism from the sub-adjacent $3$-Lie algebra $(\h,\Courant{\cdot,\cdot,\cdot})$ to $(\g,[\cdot,\cdot,\cdot]_{\g})$, where the $3$-Lie bracket $\Courant{\cdot,\cdot,\cdot}$ is defined by
    \begin{eqnarray*}
    \Courant{u_1,u_2,u_3}=\rho(Tu_1,Tu_2)u_3+\rho(Tu_2,Tu_3)u_1+\rho(Tu_3,Tu_1)u_2+\lambda[u_1,u_2,u_3]_\h.
    \end{eqnarray*}
\end{itemize}
\end{thm}

\begin{proof}
{\rm (i)} For any $u_1,u_2,u_3\in \h,$ it is obvious that
\begin{eqnarray*}
&&\{u_1,u_2,u_3\}=\rho(Tu_1,Tu_2)u_3=-\rho(Tu_2,Tu_1)u_3=-\{u_2,u_1,u_3\}.
\end{eqnarray*}
Furthermore, for $u_1,u_2,u_3,u_4,u_5\in \h$, by \eqref{representation-1}, \eqref{eq:rRB} and \eqref{3-post-Lie}, we have
\begin{eqnarray*}
&&\{u_1,u_2,\{u_3,u_4,u_5\}\}-\{u_3,u_4,\{u_1,u_2,u_5\}\}\\
&&\quad -\{\Courant{u_1,u_2,u_3},u_4,u_5\}-\{u_3,\Courant{u_1,u_2,u_4},u_5\}\\
&=&\{u_1,u_2,\rho(Tu_3,Tu_4)u_5\}-\{u_3,u_4,\rho(Tu_1,Tu_2)u_5\}\\
&&-\{\rho(Tu_1,Tu_2)u_3+\rho(Tu_2,Tu_3)u_1+\rho(Tu_3,Tu_1)u_2+\lambda[u_1,u_2,u_3]_{\h},u_4,u_5\}\\
&&-\{u_3,\rho(Tu_1,Tu_2)u_4+\rho(Tu_2,Tu_4)u_1+\rho(Tu_4,Tu_1)u_2+\lambda[u_1,u_2,u_3]_{\h},u_5\}\\
&=&\rho(Tu_1,Tu_2)\rho(Tu_3,Tu_4)u_5-\rho(Tu_3,Tu_4)\rho(Tu_1,Tu_2)u_5\\
&&-\rho\Big(T\rho(Tu_1,Tu_2)u_3+T\rho(Tu_2,Tu_3)u_1+T\rho(Tu_3,Tu_1)u_2+T\lambda[u_1,u_2,u_3]_{\h},Tu_4\Big)u_5\\
&&-\rho\Big(Tu_3,T\rho(Tu_1,Tu_2)u_4+T\rho(Tu_2,Tu_4)u_1+T\rho(Tu_4,Tu_1)u_2+T\lambda[u_1,u_2,u_4]_{\h}\Big)u_5\\
&=&0.
\end{eqnarray*}
This implies that \eqref{3-Post-Lie-2} in Definition \ref{defi-3-post-Lie-algebra} holds.
Similarly, by \eqref{representation-2}, we can verify that \eqref{3-Post-Lie-3} holds.
\emptycomment{
\begin{eqnarray*}
&&\{\Courant{u_1,u_2,u_3},u_4,u_5\}-\circlearrowleft_{u_1,u_2,u_3}\{u_1,u_2,\{u_3,u_4,u_5\}\}\\
&=&\rho([Tu_1,Tu_2,Tu_3]_{\g},Tu_4)u_5-\rho(Tu_1,Tu_2)\rho(Tu_3,Tu_4)u_5\\
&&-\rho(Tu_2,Tu_3)\rho(Tu_1,Tu_4)u_5-\rho(Tu_3,Tu_1)\rho(Tu_2,Tu_4)u_5\\
&=&0.
\end{eqnarray*}}
Moreover, by \eqref{eq:action-1} and \eqref{eq:action-2}, \eqref{3-Post-Lie-4} and \eqref{3-Post-Lie-5} hold, too.
Hence $(\h,\{\cdot,\cdot,\cdot\},[\cdot,\cdot,\cdot])$ is a $3$-post-Lie algebra.

{\rm (ii)} Note that the sub-adjacent $3$-Lie algebra of the above $3$-post-Lie algebra $(\h,\{\cdot,\cdot,\cdot\},[\cdot,\cdot,\cdot])$ is exactly the $3$-Lie algebra $(\h,[\cdot,\cdot,\cdot]_{T})$ given in Corollary \ref{new3-liealg}. Then the result follows.
\end{proof}

\begin{pro}\label{pro:mor-RB-3-post-Lie}
Let $T$ and $T'$ be two relative Rota-Baxter operators of weight $\lambda$ from a $3$-Lie algebra $(\h,[\cdot,\cdot,\cdot]_\h)$ to a $3$-Lie algebra $(\g,[\cdot,\cdot,\cdot]_\g)$ with respect to an action $\rho$. Let $(\h,\{\cdot,\cdot,\cdot\},[\cdot,\cdot,\cdot])$ and $(\h,\{\cdot,\cdot,\cdot\}',[\cdot,\cdot,\cdot]')$ be the induced $3$-post-Lie algebras and $(\psi_\g, \psi_\h)$ be a homomorphism from $T'$ to $T$. Then $\psi_\h$ is a homomorphism from the $3$-post-Lie algebra $(\h,\{\cdot,\cdot,\cdot\},[\cdot,\cdot,\cdot])$ to the $3$-post-Lie algebra $(\h,\{\cdot,\cdot,\cdot\}',[\cdot,\cdot,\cdot]')$.
\end{pro}
\begin{proof}
For all $u,v,w\in \h,$ by \eqref{condition-1}-\eqref{condition-2} and \eqref{twisted-post}, we have
\begin{eqnarray*}
\psi_\h(\{u,v,w\})&=&\psi_\h(\rho(Tu,Tv)w)\\&=&\rho(\psi_\g(Tu),\psi_\g(Tv))\psi_\h(w)\\
&=&\rho(T'\psi_\h(u),T'\psi_\h(v))\psi_\h(w)\\&=&\{\psi_\h(u),\psi_\h(v),\psi_\h(w)\}',\\
 \psi_\h([u,v,w])&=&\psi_\h\lambda[u,v,w]_{\h}\\&=&\lambda[\psi_\h(u),\psi_\h(v),\psi_\h(w)]_{\h}\\
 &=&[\psi(u),\psi(v),\psi(w)]',
\end{eqnarray*}
which implies that $\psi_\h$ is a homomorphism between the induced 3-post-Lie algebras.
\end{proof}

Thus,  Theorem \ref{construct-3-Post-Lie algebra} can be enhanced to a functor from the category of relative Rota-Baxter operators of weight $\lambda$ on $3$-Lie algebras to the category of $3$-post-Lie algebras.

\section{Maurer-Cartan characterization of relative Rota-Baxter operators of weight $\lambda$}\label{sec:four}
In this section, given an action of 3-Lie algebras, we construct an $L_{\infty}$-algebra  whose Maurer-Cartan elements are relative Rota-Baxter operators of weight $\lambda$ on $3$-Lie algebras. Then we obtain the $L_{\infty}$-algebra that controls deformations of relative Rota-Baxter operators of weight $\lambda$ on $3$-Lie algebras. This fact can be viewed as certain justification of relative Rota-Baxter operators of weight $\lambda$ on $3$-Lie algebras being interesting structures.

\begin{defi}
An {\em  $L_\infty$-algebra} is a $\mathbb Z$-graded vector space $\g=\oplus_{k\in\mathbb Z}\g^k$ equipped with a collection $(k\ge 1)$ of linear maps $l_k:\otimes^k\g\lon\g$ of degree $1$ with the property that, for any homogeneous elements $x_1,\cdots,x_n\in \g$, we have
\begin{itemize}\item[\rm(i)]
{\em (graded symmetry)} for every $\sigma\in\mathbb S_{n}$,
\begin{eqnarray*}
l_n(x_{\sigma(1)},\cdots,x_{\sigma(n-1)},x_{\sigma(n)})=\varepsilon(\sigma)l_n(x_1,\cdots,x_{n-1},x_n),
\end{eqnarray*}
\item[\rm(ii)] {\em (generalized Jacobi Identity)} for all $n\ge 1$,
\begin{eqnarray*}\label{sh-Lie}
\sum_{i=1}^{n}\sum_{\sigma\in \mathbb S_{(i,n-i)} }\varepsilon(\sigma)l_{n-i+1}(l_i(x_{\sigma(1)},\cdots,x_{\sigma(i)}),x_{\sigma(i+1)},\cdots,x_{\sigma(n)})=0.
\end{eqnarray*}
\end{itemize}
\end{defi}

\begin{defi}
 A {\bf  Maurer-Cartan element} of an $L_\infty$-algebra $(\g=\oplus_{k\in\mathbb Z}\g^k,\{l_i\}_{i=1}^{+\infty})$ is an element $\alpha\in \g^0$ satisfying the Maurer-Cartan equation
\begin{eqnarray}\label{MC-equationL}
\sum_{n=1}^{+\infty} \frac{1}{n!}l_n(\alpha,\cdots,\alpha)=0.
\end{eqnarray}
\end{defi}
Let $\alpha$ be a Maurer-Cartan element of an $L_\infty$-algebra $(\g,\{l_i\}_{i=1}^{+\infty})$. For all $k\geq1$ and $x_1,\cdots,x_n\in \g,$
define a series of linear maps $l_k^\alpha:\otimes^k\g\lon\g$ of degree $1$ by
\begin{eqnarray}
 l^{\alpha}_{k}(x_1,\cdots,x_k)=\sum^{+\infty}_{n=0}\frac{1}{n!}l_{n+k}\{\underbrace{\alpha,\cdots,\alpha}_n,x_1,\cdots,x_k\}.
\end{eqnarray}

\begin{thm}{\rm (\cite{Getzler})}\label{thm:twist}
With the above notations, $(\g,\{l^{\alpha}_i\}_{i=1}^{+\infty})$ is an $L_{\infty}$-algebra, obtained from the $L_\infty$-algebra $(\g,\{l_i\}_{i=1}^{+\infty})$ by twisting with the Maurer-Cartan element $\alpha$. Moreover, $\alpha+\alpha'$ is a Maurer-Cartan element of $(\g,\{l_i\}_{i=1}^{+\infty})$ if and only if $\alpha'$ is a Maurer-Cartan element of the twisted $L_{\infty}$-algebra  $(\g,\{l^{\alpha}_i\}_{i=1}^{+\infty})$.
\end{thm}

In \cite{Vo}, Th. Voronov developed the theory of higher derived brackets, which is a useful tool to construct explicit $L_\infty$-algebras.
\begin{defi}{\rm (\cite{Vo})}
A {\bf $V$-data} consists of a quadruple $(L,F,\huaP,\Delta)$, where
\begin{itemize}
\item[$\bullet$] $(L,[\cdot,\cdot])$ is a graded Lie algebra,
\item[$\bullet$] $F$ is an abelian graded Lie subalgebra of $(L,[\cdot,\cdot])$,
\item[$\bullet$] $\huaP:L\lon L$ is a projection, that is $\huaP\circ \huaP=\huaP$, whose image is $F$ and kernel is a  graded Lie subalgebra of $(L,[\cdot,\cdot])$,
\item[$\bullet$] $\Delta$ is an element in $  \ker(\huaP)^1$ such that $[\Delta,\Delta]=0$.
\end{itemize}

\begin{thm}{\rm (\cite{Vo})}\label{thm:db}
Let $(L,F,\huaP,\Delta)$ be a $V$-data. Then $(F,\{{l_k}\}_{k=1}^{+\infty})$ is an $L_\infty$-algebra, where
\begin{eqnarray}\label{V-shla}
l_k(a_1,\cdots,a_k)=\huaP\underbrace{[\cdots[[}_k\Delta,a_1],a_2],\cdots,a_k],\quad\mbox{for homogeneous}~   a_1,\cdots,a_k\in F.
\end{eqnarray}
We call $\{{l_k}\}_{k=1}^{+\infty}$ the {\bf higher derived brackets} of the $V$-data $(L,F,\huaP,\Delta)$. 
\end{thm}

\end{defi}
Let $\g$ be a vector space. We consider the graded vector space $$C^*(\g,\g)=\oplus_{n\ge 0}C^n(\g,\g)=\oplus_{n\ge 0}\Hom (\underbrace{\wedge^{2} \g\otimes \cdots\otimes \wedge^{2}\g}_{n}\wedge \g, \g).$$

\begin{thm}{\rm (\cite{NR bracket of n-Lie})}\label{thm:MCL}
The graded vector space  $C^*(\g,\g)$ equipped with the  graded commutator bracket
\begin{eqnarray}\label{3-Lie-bracket}
[P,Q]_{\Ri}=P{\circ}Q-(-1)^{pq}Q{\circ}P,\quad \forall~ P\in C^{p}(\g,\g),Q\in C^{q}(\g,\g),
\end{eqnarray}
is a graded Lie algebra, where $P{\circ}Q\in C^{p+q}(\g,\g)$ is defined by

\begin{small}
\begin{equation*}
\begin{aligned}
&(P{\circ}Q)(\mathfrak{X}_1,\cdots,\mathfrak{X}_{p+q},x)\\
=&\sum_{k=1}^{p}(-1)^{(k-1)q}\sum_{\sigma\in \mathbb S(k-1,q)}(-1)^\sigma P\Big(\mathfrak{X}_{\sigma(1)},\cdots,\mathfrak{X}_{\sigma(k-1)},
Q\big(\mathfrak{X}_{\sigma(k)},\cdots,\mathfrak{X}_{\sigma(k+q-1)},x_{k+q}\big)\wedge y_{k+q},\mathfrak{X}_{k+q+1},\cdots,\mathfrak{X}_{p+q},x\Big)\\
&+\sum_{k=1}^{p}(-1)^{(k-1)q}\sum_{\sigma\in \mathbb S(k-1,q)}(-1)^\sigma P\Big(\mathfrak{X}_{\sigma(1)},\cdots,\mathfrak{X}_{\sigma(k-1)},x_{k+q}\wedge
Q\big(\mathfrak{X}_{\sigma(k)},\cdots,\mathfrak{X}_{\sigma(k+q-1)},y_{k+q}\big),\mathfrak{X}_{k+q+1},\cdots,\mathfrak{X}_{p+q},x\Big)\\
&+\sum_{\sigma\in \mathbb S(p,q)}(-1)^{pq}(-1)^\sigma P\Big(\mathfrak{X}_{\sigma(1)},\cdots,\mathfrak{X}_{\sigma(p)},
Q\big(\mathfrak{X}_{\sigma(p+1)},\cdots,\mathfrak{X}_{\sigma(p+q-1)},\mathfrak{X}_{\sigma(p+q)},x\big)\Big),\\
\end{aligned}
\end{equation*}
\end{small}
 for all $\mathfrak{X}_{i}=x_i\wedge y_i\in \wedge^2 \g$, $~i=1,2,\cdots,p+q$ and $x\in\g.$

  Moreover,  $\mu:\wedge^3\g\longrightarrow\g$ is a $3$-Lie bracket if and only if $[\mu,\mu]_{\Ri}=0$, i.e.~$\mu$ is a Maurer-Cartan element of the graded Lie algebra $(C^*(\g,\g),[\cdot,\cdot]_{\Ri})$.
  \end{thm}

Let $\rho: \wedge^2\g\rightarrow \gl(\h)$ be an action of a $3$-Lie algebra $(\g,[\cdot,\cdot,\cdot]_{\g})$ on a $3$-Lie algebra $(\h,[\cdot,\cdot,\cdot]_{\h})$. For convenience, we use $\pi:\wedge^3\g\rightarrow\g$ to indicate the $3$-Lie bracket $[\cdot,\cdot,\cdot]_\g$ and $\mu:\wedge^3\h\rightarrow\h$ to indicate the $3$-Lie bracket $[\cdot,\cdot,\cdot]_\h$. In the sequel, we use  $\pi+\rho+\lambda\mu$ to denote the element in $\Hom(\wedge^3(\g\oplus \h),\g\oplus\h)$  given by
 \begin{equation}
(\pi+\rho+\lambda\mu)(x+u,y+v,z+w)=[x,y,z]_{\g}+\rho(x,y)w+\rho(y,z)u+\rho(z,x)v+\lambda[u,v,w]_\h,
\end{equation}
  for all $x,y,z\in\g,~u,v,w\in \h,~\lambda\in \mathbb K.$ Note that the right hand side is exactly the semidirect product 3-Lie algebra structure given in Proposition \ref{pro:semi}.
 Therefore by Theorem \ref{thm:MCL}, we have $$[\pi+\rho+\lambda\mu,\pi+\rho+\lambda\mu]_{\Ri}=0.$$

\begin{pro}
Let $\rho: \wedge^2\g\rightarrow \gl(\h)$ be an action of a $3$-Lie algebra $(\g,[\cdot,\cdot,\cdot]_{\g})$ on a $3$-Lie algebra $(\h,[\cdot,\cdot,\cdot]_{\h})$. Then we have a $V$-data $(L,F,\huaP,\Delta)$ as follows:
\begin{itemize}
\item[$\bullet$] the graded Lie algebra $(L,[\cdot,\cdot])$ is given by $(C^*(\g\oplus \h,\g\oplus \h),[\cdot,\cdot]_{\Ri})$;
\item[$\bullet$] the abelian graded Lie subalgebra $F$ is given by
$$
F=C^*(\h,\g)=\oplus_{n\geq 0}C^{n}(\h,\g)=\oplus_{n\geq 0}\Hom(\underbrace{\wedge^{2} \h\otimes \cdots\otimes \wedge^{2}\h}_{n}\wedge \h, \g);
$$
\item[$\bullet$] $\huaP:L\lon L$ is the projection onto the subspace $F$;
\item[$\bullet$] $\Delta=\pi+\rho+\lambda\mu.$
\end{itemize}
Consequently, we obtain an $L_\infty$-algebra $(C^*(\h,\g),l_1,l_3)$, where
\begin{eqnarray*}
l_1(P)&=&\huaP[\pi+\rho+\lambda\mu,P]_{\Ri},\\
l_3(P,Q,R)&=&\huaP[[[\pi+\rho+\lambda\mu,P]_{\Ri},Q]_{\Ri},R]_{\Ri},
\end{eqnarray*}
for all $P\in C^m(\h,\g),Q\in C^n(\h,\g)$ and $R\in C^k(\h,\g).$
\end{pro}
\begin{proof}
By Theorem \ref{thm:db}, $(F,\{l_k\}^{\infty}_{k=1})$ is an $L_{\infty}$-algebra, where $l_k$ is given by \eqref{V-shla}.
It is obvious that $\Delta=\pi+\rho+\lambda\mu \in \ker(\huaP)^{1}$. For all $P\in C^m(\h,\g),Q\in C^n(\h,\g)$ and $R\in C^k(\h,\g)$, we have
\begin{eqnarray*}
&&[[\pi+\rho+\lambda\mu,P]_{\Ri},Q]_{\Ri}\in\ker(\huaP),
\end{eqnarray*}
which implies that $l_2=0$. Similarly,  we have $l_k=0,$ when $k\geq 4$.
Therefore, the graded vector space $C^*(\h,\g)$ is an $L_{\infty}$-algebra with nontrivial $l_1$, $l_3,$ and other maps are trivial.
\end{proof}

\begin{thm}\label{thm:infty-algebra}
 Let $\rho: \wedge^2\g\rightarrow \gl(\h)$ be an action of a $3$-Lie algebra $(\g,[\cdot,\cdot,\cdot]_{\g})$ on a $3$-Lie algebra $(\h,[\cdot,\cdot,\cdot]_{\h})$. Then Maurer-Cartan elements of the $L_{\infty}$-algebra $(C^*(\h,\g),l_1,l_3)$ are precisely relative Rota-Baxter operators of weight $\lambda$ from the $3$-Lie algebra $(\h,[\cdot,\cdot,\cdot]_\h)$ to the $3$-Lie algebra $(\g,[\cdot,\cdot,\cdot]_\g)$ with respect to the action $\rho$.
\end{thm}

\begin{proof}
It is straightforward to deduce that
\begin{eqnarray*}
&&\huaP[\pi+\rho+\lambda\mu,T]_{\Ri}(u,v,w)=-T\lambda\mu(u,v,w),\\
&& \huaP[[[\pi+\rho+\lambda\mu,T]_{\Ri},T]_{\Ri},T]_{\Ri}(u,v,w)\\
&=&6\Big(\pi(Tu,Tv,Tw)-T(\rho(Tu,Tv)w+\rho(Tv,Tw)u+\rho(Tw,Tu)v)\Big).
\end{eqnarray*}
Let $T$ be a Maurer-Cartan element of the $L_{\infty}$-algebra $(C^*(\h,\g),l_1,l_3)$. We have
\begin{eqnarray*}
&&\sum_{n=1}^{+\infty} \frac{1}{n!}l_n(T,\cdots,T)(u,v,w)\\
&=&\huaP[\pi+\rho+\lambda\mu,T]_{\Ri}(u,v,w)+\frac{1}{3!}\huaP[[[\pi+\rho+\lambda\mu,T]_{\Ri},T]_{\Ri},T]_{\Ri}(u,v,w)\\
&=&\pi(Tu,Tv,Tw)-T\Big(\rho(Tu,Tv)w+\rho(Tv,Tw)u+\rho(Tw,Tu)v+\lambda\mu(u,v,w)\Big)\\
&=&0,
\end{eqnarray*}
which implies that $T$ is a relative Rota-Baxter operator of weight $\lambda$ from the $3$-Lie algebra $(\h,[\cdot,\cdot,\cdot]_\h)$ to the $3$-Lie algebra $(\g,[\cdot,\cdot,\cdot]_\g)$ with respect to the action $\rho$.
\end{proof}

\begin{pro}
Let $T$ be a relative Rota-Baxter operator of weight $\lambda$ from a $3$-Lie algebra $\h$ to a $3$-Lie algebra $\g$ with respect to an action $\rho$.
 Then $C^*(\h,\g)$ carries a twisted $L_{\infty}$-algebra structure as following:
\begin{eqnarray}
\label{twist-rota-baxter-1}l_1^{T}(P)&=&l_1(P)+\frac{1}{2}l_3(T,T,P),\\
\label{twist-rota-baxter-2}l_2^{T}(P,Q)&=&l_3(T,P,Q),\\
\label{twist-rota-baxter-3}l_3^{T}(P,Q,R)&=&l_3(P,Q,R),\\
l^T_k&=&0,\,\,\,\,k\ge4,
\end{eqnarray}
where $P\in C^m(\h,\g),Q\in C^n(\h,\g)$ and $R\in C^k(\h,\g)$.
\end{pro}
\begin{proof}
 Since $T$ is a Maurer-Cartan element of the $L_{\infty}$-algebra  $(C^*(\h,\g),l_1,l_3)$, by Theorem~\ref{thm:twist}, we have the conclusions.
\end{proof}

The above $L_\infty$-algebra controls deformations of relative Rota-Baxter operators of weight $\lambda$ on $3$-Lie algebras.

\begin{thm}\label{thm:deformation}
Let $T:\h\rightarrow\g$ be a relative Rota-Baxter operator of weight $\lambda$ from a $3$-Lie algebra $(\h,[\cdot,\cdot,\cdot]_\h)$ to a $3$-Lie algebra $(\g,[\cdot,\cdot,\cdot]_\g)$ with respect to an action $\rho$. Then for a linear map $T':\h\rightarrow \g$, $T+T'$ is a relative Rota-Baxter operator if and only if $T'$ is a Maurer-Cartan element of the twisted $L_\infty$-algebra $(C^*(\h,\g),l_1^{T},l_2^{T},l_3^{T})$, that is $T'$ satisfies the Maurer-Cartan equation:
$$
l_1^{T}(T')+\frac{1}{2}l_2^{T}(T',T')+\frac{1}{3!}l_3^{T}(T',T',T')=0.
$$
\end{thm}
\begin{proof}
  By Theorem \ref{thm:infty-algebra}, $T+T'$ is a relative Rota-Baxter operator if and only if
  $$l_1(T+T')+\frac{1}{3!}(T+T',T+T',T+T')=0.$$
Applying $l_1(T)+\frac{1}{3!}l_3(T,T,T)=0,$ the above condition is equivalent to
  $$l_1(T')+\frac{1}{2}l_3(T,T,T')+\frac{1}{2}l_3(T,T',T')+\frac{1}{6}l_3(T',T',T')=0.$$
That is, $l_1^{T}(T')+\frac{1}{2}l_2^{T}(T',T')+\frac{1}{3!}l_3^{T}(T',T',T')=0,$
\vspace{1mm}which implies that $T'$ is a Maurer-Cartan element of the twisted $L_\infty$-algebra $(C^*(\h,\g),l_1^{T},l_2^{T},l_3^{T})$.
\end{proof}

\section{Cohomologies of relative Rota-Baxter operators and infinitesimal deformations}\label{sec:five}

Let $(V;\rho)$ be a representation of a $3$-Lie algebra  $(\g,[\cdot,\cdot,\cdot]_{\g})$. Denote by
$$\mathfrak C_{\Li}^{n}(\g;V):=\Hom (\underbrace{\wedge^{2} \g\otimes \cdots\otimes \wedge^{2}\g}_{(n-1)}\wedge \g,V),\quad(n\geq 1),$$
which is the space of $n$-cochains.
The coboundary operator ${\rm d}:\mathfrak C_{\Li}^{n}(\g;V)\rightarrow \mathfrak C_{\Li}^{n+1}(\g;V)$ is defined by
\begin{eqnarray*}&&
({\rm d}f)(\mathfrak{X}_1,\cdots,\mathfrak{X}_n,x_{n+1})\\
&=&\sum_{1\leq j<k\leq n}(-1)^{j} f(\mathfrak{X}_1,\cdots,\hat{\mathfrak{X}_{j}},\cdots,\mathfrak{X}_{k-1},
[x_j,y_j,x_k]_{\g}\wedge y_k\\&&+x_k\wedge[x_j,y_j,y_k]_{\g},
\mathfrak{X}_{k+1},\cdots,\mathfrak{X}_{n},x_{n+1})\\&&
+\sum_{j=1}^{n}(-1)^{j}f(\mathfrak{X}_1,\cdots,\hat{\mathfrak{X}_{j}},\cdots,\mathfrak{X}_{n},
[x_j,y_j,x_{n+1}]_{\g})\\&&
+\sum_{j=1}^{n}(-1)^{j+1}\rho(x_j,y_j)f(\mathfrak{X}_1,\cdots,\hat{\mathfrak{X}_{j}},
\cdots,\mathfrak{X}_{n},x_{n+1})\\&&
+(-1)^{n+1}\Big(\rho(y_n,x_{n+1})f(\mathfrak{X}_1,\cdots,\mathfrak{X}_{n-1},x_n)+\rho(x_{n+1},x_n)f(\mathfrak{X}_1,\cdots,\mathfrak{X}_{n-1},y_n)\Big),
\end{eqnarray*}
for all$~\mathfrak{X}_{i}=x_{i}\wedge y_{i}\in \wedge^{2}\g,~i=1,2,\cdots,n~and~x_{n+1}\in \g.$ It was proved in \cite{Casas,Takhtajan1} that ${\rm d}\circ{\rm d}=0.$ Thus, $(\oplus_{n=1}^{+\infty}\mathfrak C_{\Li}^{n}(\g;V),{\rm d})$ is a cochain complex.

\begin{defi}
The {\bf cohomology} of the $3$-Lie algebra $\g$ with coefficients in $V$ is the cohomology of the cochain complex $(\oplus_{n=1}^{+\infty} \mathfrak C_{\Li}^{n}(\g;V),{\rm d})$. Denote by $\huaZ_{\Li}^{n}(\g;V)$ and $\huaB_{\Li}^{n}(\g;V)$
the set of $n$-cocycles and the set of $n$-coboundaries, respectively. The $n$-th cohomology group is defined by
\begin{eqnarray*}
\huaH_{\Li}^{n}(\g;V)=\huaZ_{\Li}^{n}(\g;V)/\huaB_{\Li}^{n}(\g;V).
\end{eqnarray*}
\end{defi}

\subsection{Cohomologies of relative Rota-Baxter operators of weight $\lambda$ on 3-Lie algebras}

In this subsection, we construct a representation of the $3$-Lie algebra $(\h,[\cdot,\cdot,\cdot]_T)$ on the vector space $\g$ from a relative Rota-Baxter operator $T:\h\to\g$ of weight $\lambda$ and define the cohomologies of relative Rota-Baxter operators of weight $\lambda$ on $3$-Lie algebras.

\begin{lem}\label{relative-RB-operator-representation}
Let $T:\h\rightarrow\g$ be a relative Rota-Baxter operator of weight $\lambda$ from a $3$-Lie algebra $(\h,[\cdot,\cdot,\cdot]_\h)$ to a $3$-Lie algebra $(\g,[\cdot,\cdot,\cdot]_\g)$ with respect to an action $\rho$. Define $\varrho: \wedge^2\h\rightarrow\gl(\g)$ by
 \begin{equation}\label{eq:relative-RB-operator-representation}
 \quad\varrho(u,v)(x)=[Tu,Tv,x]_{\g}-T\Big(\rho(x,Tu)v+\rho(Tv,x)u\Big),
  \end{equation}
for all $x\in \g,u,v\in \h.$  Then $(\g;\varrho)$ is a representation of the descendent $3$-Lie algebra $(\h,[\cdot,\cdot,\cdot]_T)$.
\end{lem}
\begin{proof}
By a direct calculation using \eqref{eq:jacobi1}-\eqref{eq:rRB} and \eqref{eq:des3-Lieb},
for all $u_i\in \h,1\leq i\leq 4, x\in \g,$ we have
\begin{eqnarray*}
&&\Big(\varrho(u_1,u_2)\varrho(u_3,u_4)-\varrho(u_3,u_4)\varrho(u_1,u_2) -\varrho([u_1,u_2,u_3]_{T},u_4)+\varrho([u_1,u_2,u_4]_{T},u_3)\Big)(x)\\
&=&\varrho(u_1,u_2)\Big([Tu_3,Tu_4,x]_{\g}-T\rho(x,Tu_3)u_4-T\rho(Tu_4,x)u_3\Big)\\
&&-\varrho(u_3,u_4)\Big([Tu_1,Tu_2,x]_{\g}-T\rho(x,Tu_1)u_2-T\rho(Tu_2,x)u_1\Big)\\
&&-[[Tu_1,Tu_2,Tu_3]_{\g},Tu_4,x]_{\g}+T\Big(\rho(x,[Tu_1,Tu_2,Tu_3]_{\g})u_4+\rho(Tu_4,x)[u_1,u_2,u_3]_{T}\Big)\\
&&+[[Tu_1,Tu_2,Tu_4]_{\g},Tu_3,x]_{\g}-T\Big(\rho(x,[Tu_1,Tu_2,Tu_4]_{\g})u_3+\rho(Tu_3,x)[u_1,u_2,u_4]_{T}\Big)\\
&=&-[Tu_1,Tu_2,T\rho(x,Tu_3)u_4]_{\g}-[Tu_1,Tu_2,T\rho(Tu_4,x)u_3]_{\g}\\
&&-T\Big(\rho([Tu_3,Tu_4,x]_{\g},Tu_1)u_2+\rho(Tu_2,[Tu_3,Tu_4,x]_{\g})u_1\Big)\\
&&+T\Big(\rho(T\rho(x,Tu_3)u_4,Tu_1)u_2+\rho(Tu_2,T\rho(x,Tu_3)u_4)u_1\Big)\\
&&+T\Big(\rho(T\rho(Tu_4,x)u_3,Tu_1)u_2+\rho(Tu_2,T\rho(Tu_4,x)u_3)u_1\Big)\\
&&+[Tu_3,Tu_4,T\rho(Tu_2,x)u_1]_{\g}+[Tu_3,Tu_4,T\rho(x,Tu_1)u_2]_{\g}\\
&&+T\Big(\rho([Tu_1,Tu_2,x]_{\g},Tu_3)u_4+\rho(Tu_4,[Tu_1,Tu_2,x]_{\g})u_3\Big)\\
&&-T\Big(\rho(T\rho(x,Tu_1)u_2,Tu_3)u_4+\rho(Tu_4,T\rho(x,Tu_1)u_2)u_3\Big)\\
&&-T\Big(\rho(T\rho(Tu_2,x)u_1,Tu_3)u_4+\rho(Tu_4,T\rho(Tu_2,x)u_1)u_3\Big)\\
&&+T\Big(\rho(x,[Tu_1,Tu_2,Tu_3]_{\g})u_4+\rho(Tu_4,x)\rho(Tu_1,Tu_2)u_3\\
&&+\rho(Tu_4,x)\rho(Tu_2,Tu_3)u_1+\rho(Tu_4,x)\rho(Tu_3,Tu_1)u_2+\lambda\rho(Tu_4,x)[u_1,u_2,u_3]_{\h}\Big)\\
&&-T\Big(\rho(x,[Tu_1,Tu_2,Tu_4]_{\g})u_3+\rho(Tu_3,x)\rho(Tu_1,Tu_2)u_4\\
&&+\rho(Tu_3,x)\rho(Tu_2,Tu_4)u_1+\rho(Tu_3,x)\rho(Tu_4,Tu_1)u_2+\lambda\rho(Tu_3,x)[u_1,u_2,u_4]_{\h}\Big)\\
&=&0,
\end{eqnarray*}
and
\begin{eqnarray*}
&&\Big(\varrho([u_1,u_2,u_3]_{T},u_4)-\varrho(u_1,u_2)\varrho(u_3,u_4)-\varrho(u_2,u_3)\varrho(u_1,u_4)-\varrho(u_3,u_1)\varrho(u_2,u_4)\Big)(x)\\
&=&[[Tu_1,Tu_2,Tu_3]_{\g},Tu_4,x]_{\g}-T\Big(\rho(x,[Tu_1,Tu_2,Tu_3]_{\g})u_4+\rho(Tu_4,x)[u_1,u_2,u_3]_{T}\Big)\\
&&-\varrho(u_1,u_2)\Big([Tu_3,Tu_4,x]_{\g}-T\rho(x,Tu_3)u_4-T\rho(Tu_4,x)u_3\Big)\\
&&-\varrho(u_2,u_3)\Big([Tu_1,Tu_4,x]_{\g}-T\rho(x,Tu_1)u_4-T\rho(Tu_4,x)u_1\Big)\\
&&-\varrho(u_3,u_1)\Big([Tu_2,Tu_4,x]_{\g}-T\rho(x,Tu_2)u_4-T\rho(Tu_4,x)u_2\Big)\\
&=&-T\Big(\rho(x,[Tu_1,Tu_2,Tu_3]_{\g})u_4+\lambda\rho(Tu_4,x)[u_1,u_2,u_3]_{\h}\Big)\\
&&-T\rho(Tu_4,x)\Big(\rho(Tu_1,Tu_2)u_3+\rho(Tu_2,Tu_3)u_1+\rho(Tu_3,Tu_1)u_2\Big)\\
&&+[Tu_1,Tu_2,T\rho(x,Tu_3)u_4+T\rho(Tu_4,x)u_3]\\
&&+T\Big(\rho([Tu_3,Tu_4,x]_{\g},Tu_1)u_2+\rho(Tu_2,[Tu_3,Tu_4,x]_{\g})u_1\Big)\\
&&-T\Big(\rho(T\rho(x,Tu_3)u_4,Tu_1)u_2+\rho(Tu_2,T\rho(x,Tu_3)u_4)u_1\Big)\\
&&-T\Big(\rho(T\rho(Tu_4,x)u_3,Tu_1)u_2+\rho(Tu_2,T\rho(Tu_4,x)u_3)u_1\Big)\\
&&+[Tu_2,Tu_3,T\rho(x,Tu_1)u_4+T\rho(Tu_4,x)u_1]\\
&&+T\Big(\rho([Tu_1,Tu_4,x]_{\g},Tu_2)u_3+\rho(Tu_3,[Tu_1,Tu_4,x]_{\g})u_2\Big)\\
&&-T\Big(\rho(T\rho(x,Tu_1)u_4,Tu_2)u_3+\rho(Tu_3,T\rho(x,Tu_1)u_4)u_2\Big)\\
&&-T\Big(\rho(T\rho(Tu_4,x)u_1,Tu_2)u_3+\rho(Tu_3,T\rho(Tu_4,x)u_1)u_2\Big)\\
&&+[Tu_3,Tu_1,T\rho(x,Tu_2)u_4+T\rho(Tu_4,x)u_2]\\
&&+T\Big(\rho([Tu_2,Tu_4,x]_{\g},Tu_3)u_1+\rho(Tu_1,[Tu_2,Tu_4,x]_{\g})u_3\Big)\\
&&-T\Big(\rho(T\rho(x,Tu_2)u_4,Tu_3)u_1+\rho(Tu_1,T\rho(x,Tu_2)u_4)u_3\Big)\\
&&-T\Big(\rho(T\rho(Tu_4,x)u_2,Tu_3)u_1+\rho(Tu_1,T\rho(Tu_4,x)u_2)u_3\Big)\\
&=&0.
\end{eqnarray*}
Thus, we deduce that $(\g;\varrho)$ is a representation of the descendent $3$-Lie algebra $(\h,[\cdot,\cdot,\cdot]_{T})$.
\end{proof}

\begin{pro}\label{prohomrepalg}
Let $T$ and $T'$ be relative Rota-Baxter operators of weight $\lambda$ from a $3$-Lie algebra $(\h,[\cdot,\cdot,\cdot]_\h)$ to a $3$-Lie algebra $(\g,[\cdot,\cdot,\cdot]_\g)$ with respect to an action $\rho$. Let $(\psi_\g, \psi_\h)$ be a homomorphism from $T$ to $T'$.

\begin{itemize}
  \item[{\rm(i)}]  $\psi_\h$ is also a $3$-Lie algebra homomorphism from
the descendent $3$-Lie algebra $(\h, [\cdot,\cdot,\cdot]_{T})$ of $T$ to the descendent $3$-Lie algebra $(\h, [\cdot,\cdot,\cdot]_{T'})$ of $T'$;

   \item[{\rm(ii)}] The  induced representation $(\g;\varrho)$ of the $3$-Lie algebra $(\h,[\cdot,\cdot,\cdot]_T)$ and the induced representation $(\g;\varrho')$ of the $3$-Lie algebra $(\h,[\cdot,\cdot,\cdot]_{T'})$ satisfy the following relation:
$$
\psi_\g\circ \varrho(u,v)=\varrho'(\psi_\h(u),\psi_\h(v)) \circ \psi_\g,\quad \forall u,v\in \h.
$$
That is, the following diagram commutes:
$$
\xymatrix{
  \g \ar[d]_{\varrho(u,v)} \ar[r]^{\psi_\g}
                & \g \ar[d]^{\varrho'(\psi_\h(u),\psi_\h(v))}  \\
    \g \ar[r]_{\psi_\g}
                & \g            .}
$$
\end{itemize}
\end{pro}
\begin{proof}

(i) It follows from Corollary \ref{cor:mor-post-3-Lie} and Proposition \ref{pro:mor-RB-3-post-Lie} directly.

(ii) By \eqref{eq:relative-RB-operator-representation}, \eqref{condition-1}-\eqref{condition-2} and the fact that $\psi_\g$ is a $3$-Lie algebra homomorphism, for all $x\in\g, u,v\in\h$, we have
\begin{eqnarray*}
\psi_\g\Big(\varrho(u,v)x\Big)&=&\psi_\g\Big([Tu,Tv,x]_{\g}-T(\rho(x,Tu)v)-T(\rho(Tv,x)u)\Big)\\
&=&[\psi_\g(Tu),\psi_\g(Tv),\psi_\g(x)]_{\g}-T'\psi_\h\Big(\rho(x,Tu)v\Big)-T'\psi_\h\Big(\rho(Tv,x)u\Big)\\
&=&[T'\psi_\h(u),T'\psi_\h(v),\psi_\g(x)]_{\g}-T'\Big(\rho(\psi_{\g}(x),T'\psi_\h(u))\psi_\h(v)\Big)\\
&&-T'\Big(\rho(T'\psi_\h(v),\psi_{\g}(x))\psi_\h(u)\Big)\\
&=&\varrho'(\psi_\h(u),\psi_\h(v))\psi_\g(x).
\end{eqnarray*}
We finish the proof.
\end{proof}

Let ${\rm d}_{T}:\mathfrak C_{\Li}^{n}(\h;\g)\rightarrow \mathfrak C_{\Li}^{n+1}(\h;\g),(n\geq1)$ be the corresponding coboundary
operator of the  descendent $3$-Lie algebra $(\h,[\cdot,\cdot,\cdot]_{T})$ with coefficients in the representation $(\g;\varrho)$.
More precisely, for all $f\in \Hom (\underbrace{\wedge^{2} \h\otimes \cdots\otimes \wedge^{2}\h}_{(n-1)}\wedge \h,\g)$, $\U_i=u_i\wedge v_i\in \wedge^2\h,~ i=1,2,\cdots,n$ and $u_{n+1}\in \h,$ we have
\begin{eqnarray*}
&&({\rm d}_{T}f)(\U_1,\cdots,\U_n,u_{n+1})\\
&=&\sum_{1\leq j<k\leq n}(-1)^jf(\U_1,\cdots,\hat{\U_j},\cdots,\U_{k-1},[u_j,v_j,u_k]_{T}\wedge v_k\\
&&+u_k\wedge[u_j,v_j,v_k]_{T},\U_{k+1},\cdots,\U_n,u_{n+1})\\
&&+\sum_{j=1}^{n}(-1)^{j}f(\U_1,\cdots,\hat{\U_j},\cdots,\U_{n},[u_j,v_j,u_{n+1}]_{T})\\
&&+\sum_{j=1}^{n}(-1)^{j+1}\varrho(u_j,v_j)f(\U_1,\cdots,\hat{\U_j},\cdots,\U_n,u_{n+1})\\
&&+(-1)^{n+1}\Big(\varrho(v_n,u_{n+1})f(\U_1,\cdots,\U_{n-1},u_n)+\varrho(u_{n+1},u_{n})f(\U_1,\cdots,\U_{n-1},v_n)\Big).
\end{eqnarray*}

It is obvious that $f\in \mathfrak C_{\Li}^{1}(\h;\g)$ is closed if and only if
\begin{eqnarray*}
&&[f(u_1),Tu_2,Tu_3]_{\g}+[Tu_1,f(u_2),Tu_3]_{\g}+[Tu_1,Tu_2,f(u_3)]_{\g}\\
&=&T\Big(\rho(Tu_2,f(u_3))u_1+\rho(f(u_3),Tu_1)u_2+\rho(f(u_1),Tu_2)u_3\Big)\\
&&+T\Big(\rho(f(u_2),Tu_3)u_1+\rho(Tu_3,f(u_1))u_2+\rho(Tu_1,f(u_2))u_3\Big)\\
&&+f\Big(\rho(Tu_2,Tu_3)u_1+\rho(Tu_3,Tu_1)u_2+\rho(Tu_1,Tu_2)u_3+\lambda[u_1,u_2,u_3]_{\h}\Big).
\end{eqnarray*}

Define $\delta:\wedge^2\g\rightarrow\Hom(\h,\g)$ by
\begin{eqnarray*}
\delta(\mathfrak{X})u=T\rho(\mathfrak{X})u-[\mathfrak{X},Tu]_{\g}, \quad \forall\mathfrak{X}\in\wedge^2\g, u\in \h.
\end{eqnarray*}
\begin{pro}\label{pro:0}
Let $T:\h\rightarrow\g$ be a relative Rota-Baxter operator of weight $\lambda$ from $\h$ to $\g$ with respect to an action  $\rho$. Then $\delta(\mathfrak{X})$ is a $1$-cocycle of the $3$-Lie algebra $(\h,[\cdot,\cdot,\cdot]_{T})$ with coefficients in $(\g;\varrho).$
\end{pro}
\begin{proof}
It follows from straightforward computations, and we omit details.
\end{proof}
We now introduce  the cohomology theory of relative Rota-Baxter operators of weight $\lambda$ on $3$-Lie algebras.

Let $T$ be a relative Rota-Baxter operator of weight $\lambda$ from a $3$-Lie algebra $(\h,[\cdot,\cdot,\cdot]_\h)$ to a $3$-Lie algebra $(\g,[\cdot,\cdot,\cdot]_\g)$ with respect to an action $\rho$.  Define the space of $n$-cochains by
\begin{eqnarray}\label{cohomology-1}
\mathfrak C_{T}^{n}(\h;\g)=
\left\{\begin{array}{rcl}
{}\mathfrak C_{\Li}^{n-1}(\h;\g),\quad n\geq 2,\\
{}\g\wedge\g,\quad n=1.
\end{array}\right.
\end{eqnarray}

Define ${\partial}:\mathfrak C_{T}^{n}(\h;\g)\rightarrow \mathfrak C_{T}^{n+1}(\h;\g)$ by
\begin{eqnarray}\label{cohomology-2}
{\partial}=
\left\{\begin{array}{rcl}
{}{\rm d}_{T},\quad n\geq 2,\\
{}\delta,\quad n=1.
\end{array}\right.
\end{eqnarray}
\begin{thm}
  $(\mathop{\oplus}\limits_{n=1}^{\infty} \mathfrak C_{T}^{n}(\h;\g),\partial)$ is a cochain complex.
\end{thm}
\begin{proof}
  It follows from Proposition \ref{pro:0} and the fact that ${\rm d}_{T} $ is the corresponding coboundary
operator of the  descendent $3$-Lie algebra $(\h,[\cdot,\cdot,\cdot]_{T})$ with coefficients in the representation $(\g;\varrho)$ directly.
\end{proof}

\begin{defi}
The cohomology of the cochain complex $(\mathop{\oplus}\limits_{n=1}^{\infty} \mathfrak C_{T}^{n}(\h;\g),\partial)$  is taken to be the {\bf cohomology for the relative Rota-Baxter operator $T$ of weight $\lambda$}. Denote the set of $n$-cocycles by $\huaZ^n_T(\h;\g),$ the set of $n$-coboundaries by $\huaB^n_T(\h;\g)$ and $n$-th cohomology group by
\begin{eqnarray}\label{cohomology-3}
\huaH^n_T(\h;\g)=\huaZ_T^n(\h;\g)/\huaB_T^n(\h;\g),\quad n\geq1.
\end{eqnarray}
\end{defi}

\begin{rmk}
 The cohomology theory for relative Rota-Baxter operators of weight $\lambda$ on $3$-Lie algebras enjoys certain functorial properties.
Let $T$ and $T'$ be relative Rota-Baxter operators of weight $\lambda$ from a $3$-Lie algebra $(\h,[\cdot,\cdot,\cdot]_\h)$ to a $3$-Lie algebra $(\g,[\cdot,\cdot,\cdot]_\g)$ with respect to an action $\rho$. Let $(\psi_\g, \psi_\h)$ be a homomorphism from $T$ to $T'$ in which $\psi_\h$ is invertible. Define a map $p:\mathfrak C_{T}^{n}(\h;\g)\rightarrow \mathfrak C_{T'}^{n}(\h;\g)$ by
\begin{eqnarray*}
p(\omega)(\U_1,\cdots,\U_{n-2},u_{n-1})=\psi_{\g}\Bigg(\omega\Big(\psi^{-1}_{\h}(u_1)\wedge\psi^{-1}_{\h}(v_1),\cdots,\psi^{-1}_{\h}(u_{n-2})\wedge\psi^{-1}_{\h}(v_{n-2}),\psi^{-1}_{\h}(u_{n-1})\Big)\Bigg),
\end{eqnarray*}
for all $\omega\in \mathfrak C_{T}^{n}(\h;\g), \U_i=u_i\wedge v_i\in \wedge^2\h,~ i=1,2,\cdots,n-2$ and $u_{n-1}\in \h.$
Then it is straightforward to deduce that $p$ is a cochain map from the cochain complex $(\mathop{\oplus}\limits_{n=2}^{\infty} \mathfrak C_{T}^{n}(\h;\g),{\rm d}_{T})$
to the cochain complex $(\mathop{\oplus}\limits_{n=2}^{\infty} \mathfrak C_{T'}^{n}(\h;\g),{\rm d}_{T'})$. Consequently, it induces a homomorphism $p_*$ from the
cohomology group $\huaH^{n}_T(\h;\g)$ to $\huaH^{n}_{T'}(\h;\g)$.
\end{rmk}

At the end of this subsection, we give the relationship between the coboundary operator ${\rm d}_{T}$
 and the differential $l_1^{T}$ defined by ~\eqref{twist-rota-baxter-1} using the Maurer-Cartan element $T$ of the $L_{\infty}$-algebra  $(C^*(\h,\g),l_1,l_3)$.
\begin{lem}\label{lem-equation-1}
Let $T:\h\rightarrow\g$ be a relative Rota-Baxter operator of weight $\lambda$ from a $3$-Lie algebra $(\h,[\cdot,\cdot,\cdot]_\h)$ to a $3$-Lie algebra $(\g,[\cdot,\cdot,\cdot]_\g)$ with respect to an action $\rho$. For all $x,y,z\in \g, u,v,w\in V,$ we have
\begin{eqnarray*}
&&[[\pi+\rho+\lambda\mu,T]_{\Ri},T]_{\Ri}(x+u,y+v,z+w)\\
&=&2\Big([Tu,Tv,z]_\g+\rho(Tu,Tv)w+[Tv,Tw,x]_\g+\rho(Tv,Tw)u+[Tw,Tu,y]_\g+\rho(Tw,Tu)v\Big)\\
&&-2T\Big(\rho(Tu,y)w+\rho(z,Tu)v+\rho(x,Tv)w+\rho(Tv,z)u+\rho(y,Tw)u+\rho(Tw,x)v\Big).
\end{eqnarray*}
\end{lem}
\begin{proof}
  It follows from straightforward computations.
\end{proof}
 \begin{thm}\label{partial-to-derivation}
 Let $T$ be a relative Rota-Baxter operator of weight $\lambda$ from a $3$-Lie algebra $(\h,[\cdot,\cdot,\cdot]_\h)$ to a $3$-Lie algebra $(\g,[\cdot,\cdot,\cdot]_\g)$ with respect to an action $\rho$. Then we have
 $$
{\rm d}_{T} f=(-1)^{n-1}l_1^{T} f,\quad \forall f\in \Hom(\underbrace{\wedge^{2} \h\otimes \cdots\otimes \wedge^{2}\h}_{n-1}\wedge\h, \g),~n=1,2,\cdots.
 $$
\end{thm}

\begin{proof}
For all $\mathfrak{U}_i=u_i\wedge v_i\in \wedge^2 \h,~ i=1,2,\cdots,n$ and $u_{n+1}\in \h$, we have
\begin{eqnarray*}
&&l_1(f)(\mathfrak{U}_1,\cdots,\mathfrak{U}_n,u_{n+1})\\
&=&[\pi+\rho+\lambda\mu,f]_{\Ri}(\mathfrak{U}_1,\cdots,\mathfrak{U}_n,u_{n+1})\\
&=&\Big((\pi+\rho+\lambda\mu)\circ f-(-1)^{n-1}f\circ(\pi+\rho+\lambda\mu)\Big)(\mathfrak{U}_1,\cdots,\mathfrak{U}_n,u_{n+1})\\
&=&(\pi+\rho+\lambda\mu)(f(\mathfrak{U}_1,\cdots,\mathfrak{U}_{n-1},u_n)\wedge v_n,u_{n+1})\\
&&+(\pi+\rho+\lambda\mu)(u_n\wedge f(\mathfrak{U}_1,\cdots,\mathfrak{U}_{n-1},v_n), u_{n+1})\\
&&+\sum_{i=1}^{n}(-1)^{n-1}(-1)^{i-1}(\pi+\rho+\lambda\mu)(\mathfrak{U}_{i},f(\mathfrak{U}_1,\cdots,\hat{\mathfrak{U}_i},\cdots,\mathfrak{U}_n,u_{n+1}))\\
&&-(-1)^{n-1}\sum_{k=1}^{n-1}\sum_{i=1}^{k}(-1)^{i+1}f\Big(\mathfrak{U}_1,\cdots,\hat{\mathfrak{U}}_{i},\cdots,\mathfrak{U}_{k}, (\pi+\rho+\lambda\mu)(\mathfrak{U}_{i},u_{k+1})\wedge v_{k+1},\mathfrak{U}_{k+2},\cdots,\mathfrak{U}_{n},u_{n+1}\Big)\\
&&-(-1)^{n-1}\sum_{k=1}^{n-1}\sum_{i=1}^{k}(-1)^{i+1}f\Big(\mathfrak{U}_1,\cdots,\hat{\mathfrak{U}}_{i},\cdots,\mathfrak{U}_{k},u_{k+1}\wedge (\pi+\rho+\lambda\mu)(\mathfrak{U}_{i},v_{k+1}),\mathfrak{U}_{k+2},\cdots,\mathfrak{U}_{n},u_{n+1}\Big)\\
&&-(-1)^{n-1}\sum_{i=1}^{n}(-1)^{i+1}f\Big(\mathfrak{U}_1,\cdots,\hat{\mathfrak{U}}_{i},\cdots,\mathfrak{U}_{n},(\pi+\rho+\lambda\mu)(\mathfrak{U}_{i},u_{n+1})\Big)\\
&=&-(-1)^{n-1}\sum_{k=1}^{n-1}\sum_{i=1}^{k}(-1)^{i+1}f\Big(\mathfrak{U}_1,\cdots,\hat{\mathfrak{U}}_{i},\cdots,\mathfrak{U}_{k}, \lambda\mu(u_i,v_i,u_{k+1})\wedge v_{k+1},\mathfrak{U}_{k+2},\cdots,\mathfrak{U}_{n},u_{n+1}\Big)\\
&&-(-1)^{n-1}\sum_{k=1}^{n-1}\sum_{i=1}^{k}(-1)^{i+1}f\Big(\mathfrak{U}_1,\cdots,\hat{\mathfrak{U}}_{i},\cdots,\mathfrak{U}_{k},u_{k+1}\wedge \lambda\mu(u_i,v_i,v_{k+1}),\mathfrak{U}_{k+2},\cdots,\mathfrak{U}_{n},u_{n+1}\Big)\\
&&-(-1)^{n-1}\sum_{i=1}^{n}(-1)^{i+1}f\Big(\mathfrak{U}_1,\cdots,\hat{\mathfrak{U}}_{i},\cdots,\mathfrak{U}_{n},\lambda\mu(u_i,v_i,u_{n+1})\Big).
\end{eqnarray*}
By Lemma \ref{lem-equation-1}, we have
\begin{eqnarray*}
&&\frac{1}{2}l_3(T,T,f)(\mathfrak{U}_1,\cdots,\mathfrak{U}_n,u_{n+1})\\
&=&[[[\pi+\rho+\lambda\mu,T]_{\Ri},T]_{\Ri},f]_{\Ri}(\mathfrak{U}_1,\cdots,\mathfrak{U}_n,u_{n+1})\\
&=&[[\pi+\rho+\lambda\mu,T]_{\Ri},T]_{\Ri}\Big(f(\mathfrak{U}_1,\cdots,\mathfrak{U}_{n-1},u_n)\wedge v_n,u_{n+1}\Big)\\
&&+[[\pi+\rho+\lambda\mu,T]_{\Ri},T]_{\Ri}\Big(u_n\wedge f(\mathfrak{U}_1,\cdots,\mathfrak{U}_{n-1},v_n),u_{n+1}\Big)\\
&&+\sum_{i=1}^{n}(-1)^{n-1}(-1)^{i-1}[[\pi+\rho+\lambda\mu,T]_{\Ri},T]_{\Ri}\Big(\mathfrak{U}_i, f(\mathfrak{U}_1,\cdots,\hat{\mathfrak{U}}_{i},\cdots,\mathfrak{U}_{n},u_{n+1})\Big)\\
&&-(-1)^{n-1}\sum_{k=1}^{n-1}\sum_{i=1}^{k}(-1)^{i+1}f\Big(\mathfrak{U}_1,\cdots,\hat{\mathfrak{U}}_{i},\cdots,\mathfrak{U}_{k},[[\pi+\rho+\lambda\mu,T]_{\Ri},T]_{\Ri}(\mathfrak{U}_{i},u_{k+1})\wedge v_{k+1}\\
&&+u_{k+1}\wedge [[\pi+\rho+\lambda\mu,T]_{\Ri},T]_{\Ri}(\mathfrak{U}_{i},v_{k+1}),\mathfrak{U}_{k+2},\cdots,\mathfrak{U}_{n},u_{n+1}\Big)\\
&&-(-1)^{n-1}\sum_{i=1}^{n}(-1)^{i+1}f\Big(\mathfrak{U}_1,\cdots,\hat{\mathfrak{U}}_{i},\cdots,\mathfrak{U}_{n},[[\pi+\rho+\lambda\mu,T]_{\Ri},T]_{\Ri}(\mathfrak{U}_{i},u_{n+1})\Big)\\
&=&\varrho(v_n,u_{n+1})f(\mathfrak{U}_1,\cdots,\mathfrak{U}_{n-1},u_n)+\varrho(u_{n+1},u_n,)f(\mathfrak{U}_1,\cdots,\mathfrak{U}_{n-1},v_n)\\
&&+\sum_{i=1}^{n}(-1)^{n-1}(-1)^{i-1}\varrho(u_i,v_i)f(\mathfrak{U}_1,\cdots,\mathfrak{U}_i,\cdots,\mathfrak{U}_{n},u_{n+1})\\
&&-(-1)^{n-1}\sum_{k=1}^{n-1}\sum_{i=1}^{k}(-1)^{i+1}f\Big(\mathfrak{U}_1,\cdots,\hat{\mathfrak{U}}_{i},\cdots,\mathfrak{U}_{k},\Big(\rho(Tu_i,Tv_i)u_{k+1}\\
&&+\rho(Tv_i,Tu_{k+1})u_i+\rho(Tu_{k+1},Tu_i)v_i\Big)\wedge v_{k+1},\mathfrak{U}_{k+2},\cdots,\mathfrak{U}_{n},u_{n+1}\Big)\\
&&-(-1)^{n-1}\sum_{k=1}^{n-1}\sum_{i=1}^{k}(-1)^{i+1}f\Big(\mathfrak{U}_1,\cdots,\hat{\mathfrak{U}}_{i},\cdots,\mathfrak{U}_{k},u_{k+1}\wedge \Big(\rho(Tu_i,Tv_i)v_{k+1}\\
&&+\rho(Tv_i,Tv_{k+1})u_i+\rho(Tv_{k+1},Tu_i)v_i\Big),\mathfrak{U}_{k+2},\cdots,\mathfrak{U}_{n},u_{n+1}\Big)\\
&&-(-1)^{n-1}\sum_{i=1}^{n}(-1)^{i+1}f\Big(\mathfrak{U}_1,\cdots,\hat{\mathfrak{U}}_{i},\cdots,\mathfrak{U}_{n},\rho(Tu_i,Tv_i)u_{n+1}+\rho(Tv_i,Tu_{n+1})u_i+\rho(Tu_{n+1},Tu_i)v_i\Big)
\end{eqnarray*}
Thus, we deduce that ${\rm d}_{T} f=(-1)^{n-1}\Big(l_1(f)+\frac{1}{2}l_3(T,T,f)\Big)$, that is  ${\rm d}_{T} f=(-1)^{n-1}l_1^{T} f$.
\end{proof}

\subsection{Infinitesimal deformations of relative Rota-Baxter operators}

In this subsection, we use the established cohomology theory to characterize  infinitesimal deformations of relative Rota-Baxter operators of weight $\lambda$ on 3-Lie algebras.

Let $(\g,[\cdot,\cdot,\cdot]_{\g})$ be a $3$-Lie algebra over $\mathbb K$ and $\mathbb K[t]$ be the polynomial ring in one variable $t.$
Then $\mathbb K[t]/(t^2)\otimes_{\mathbb K}\g$ is an $\mathbb K[t]/(t^2)$-module. Moreover, $\mathbb K[t]/(t^2)\otimes_{\mathbb K}\g$ is a $3$-Lie algebra over $\mathbb K[t]/(t^2)$, where the $3$-Lie algebra structure is defined by
\begin{eqnarray*}
[f_1(t)\otimes_{\mathbb K} x_1,f_2(t)\otimes_{\mathbb K} x_2,f_3(t)\otimes_{\mathbb K} x_3]= f_1(t)f_2(t) f_3(t)\otimes_{\mathbb K}[x_1,x_2,x_3]_{\g},
\end{eqnarray*}
for $f_{i}(t)\in \mathbb K[t]/(t^2),1\leq i\leq 3,x_1,x_2,x_3\in \g.$

In the sequel, all the vector spaces are finite dimensional vector spaces over $\mathbb K$ and we denote $f(t)\otimes_{\mathbb K} x$ by $f(t)x,$ where $f(t)\in \mathbb K[t]/(t^2).$

\begin{defi}
Let $T:\h\rightarrow \g$ be a relative Rota-Baxter operator of weight $\lambda$ from a $3$-Lie algebra $(\h,[\cdot,\cdot,\cdot]_\h)$ to a $3$-Lie algebra $(\g,[\cdot,\cdot,\cdot]_\g)$ with respect to an action $\rho$. Let
$\frkT:\h\rightarrow\g$ be a linear map. If $T_t=T+t\frkT$ is a relative Rota-Baxter operator of weight $\lambda$
modulo $t^2$, we say that $\frkT$ generates an {\bf infinitesimal deformation} of $T$.
\end{defi}

Since $T_t=T+t\frkT $ is a relative Rota-Baxter operator of weight $\lambda$ modulo $t^2$, by consider the coefficients of $t$, for any $u,v,w\in \h,$
we have
\begin{eqnarray}
\label{equivalent-1}\qquad&&[\frkT u,Tv,Tw]_{\g}+[Tu,\frkT v,Tw]_{\g}+[Tu,Tv,\frkT w]_{\g}\\
\nonumber&=&T\Big(\rho(\frkT w,Tu)v+\rho(Tv,\frkT w)u+\rho(\frkT u,Tv)w\\
\nonumber&&+\rho(Tw,\frkT u)v+\rho(\frkT v,Tw)u+\rho(Tu,\frkT v)w\Big)\\
\nonumber&&+\frkT\Big(\rho(Tu,Tv)w+\rho(Tv,Tw)u+\rho(Tw,Tu)v+\lambda[u,v,w]_{\h}\Big).
\end{eqnarray}

 Note that \eqref{equivalent-1} means that $\frkT$ is a $2$-cocycle of the relative Rota-Baxter operator $T$. Hence, $\frkT$ defines a cohomology class in $\huaH^2_{T}(\h;\g)$.

\begin{defi}
Let $T$ be a relative Rota-Baxter operator of weight $\lambda$ from a $3$-Lie algebra $(\h,[\cdot,\cdot,\cdot]_\h)$ to a $3$-Lie algebra $(\g,[\cdot,\cdot,\cdot]_\g)$ with respect to an action $\rho$. Two one-parameter infinitesimal deformations $T^1_{t}=T+t\frkT_{1}$ and  $T^2_{t}=T+t\frkT_{2}$ are said to be {\bf equivalent} if there exists $\mathfrak{X}\in\g\wedge\g$ such that $(\Id_{\g}+t\ad_{\mathfrak{X}},\Id_{\h}+t\rho(\mathfrak{X}))$ is a homomorphism from $T^1_{t}$ to $T^2_{t}$ modulo $t^2$.
In particular, an infinitesimal deformation $T_{t}=T+t\frkT_{1}$ of a relative Rota-Baxter operator $T$ is said to be {\bf trivial} if there exists $\mathfrak{X}\in \g\wedge\g$ such that $(\Id_{\g}+t\ad_{\mathfrak{X}},\Id_{\h}+t\rho(\mathfrak{X}))$ is a homomorphism from $T_{t}$ to $T$ modulo $t^2$.
\end{defi}

Let $(\Id_{\g}+t\ad_{\mathfrak{X}},\Id_{\h}+t\rho(\mathfrak{X}))$ be a homomorphism from $T^1_{t}$ to $T^2_{t}$ modulo $t^2.$
By \eqref{condition-1} we get,
\begin{equation*}
\quad (\Id_{\g}+t\ad_{\mathfrak{X}})(T+t\frkT_1)(u)=(T+t\frkT_2)(\Id_\h+t\rho(\mathfrak{X}))(u),
\end{equation*}
which implies
\begin{eqnarray}\label{Nijenhuis-element-4}
 \frkT_1(u)-\frkT_2(u)&=&T\rho(\mathfrak{X})u-[\mathfrak{X},Tu]_{\g}.
\end{eqnarray}

Now we are ready to give the main result in this section.
\begin{thm}
Let $T$ be a relative Rota-Baxter operator of weight $\lambda$ from a $3$-Lie algebra $\h$ to a $3$-Lie algebra $\g$ with respect to an action $\rho$.
If two one-parameter infinitesimal deformations $T^1_{t}=T+t\frkT_{1}$ and $T^2_{t}=T+t\frkT_{2}$ are equivalent, then $\frkT_{1}$ and $\frkT_{2}$
 are in the same cohomology class in $\huaH^2_{T}(\h;\g)$.
\end{thm}
\begin{proof}
 It is easy to see from  \eqref{Nijenhuis-element-4} that
\begin{eqnarray*}
 \frkT_1(u)&=& \frkT_2(u)+(\partial\mathfrak{X})(u),\quad \forall u\in \h,
\end{eqnarray*}
which implies that $\frkT_1$ and $\frkT_2$ are in the same cohomology class.
\end{proof}

 \end{document}